\documentclass[11pt,a4paper]{article}

\usepackage[left=1.1in, right=1.1in, top=1.0in, bottom=1.0in]{geometry}
\usepackage{amsfonts,amstext,amsmath,amssymb,amsopn,amsthm}
\usepackage{mathrsfs}
\usepackage{mathtools}
\usepackage{physics}
\usepackage[ruled,vlined]{algorithm2e}
\usepackage{todonotes}
\usepackage{enumerate}

\usepackage[colorlinks=true, 
linkcolor=blue,
pdfstartview=FitH,      
breaklinks=true,        
bookmarksopen=true,     
bookmarksnumbered=true  
]{hyperref}
\usepackage[capitalize,nameinlink]{cleveref}

\theoremstyle{plain}
\newtheorem{thm}{Theorem}

\newtheorem{proposition}[thm]{Proposition}
\newtheorem{lemma}[thm]{Lemma}
\newtheorem{corollary}[thm]{Corollary}
\newtheorem{mydef}[thm]{Definition}
\newtheorem{remark}[thm]{Remark}

\newcommand{\paren}[1]{\left(#1\right)}
\newcommand{\mymap}[3]{#1:\,#2 \to #3\,}
\DeclareMathOperator{\diam}{diam}
\DeclareMathOperator{\argmin}{argmin\,}

\DeclareMathOperator{\prox}{prox}
\DeclareMathOperator{\intr}{int\,}

\DeclareMathOperator{\Id}{Id}
\DeclareMathOperator{\Fix}{Fix}

\DeclareMathOperator{\co}{co}

\newcommand{\klam}[1]{\left\{#1\right\}}
\newcommand{\Hcal}{\mathcal{H}}
\newcommand{\Nbb}{\mathbb{N}}
\newcommand{\Ncal}{\mathcal{N}}
\newcommand{\Rbb}{\mathbb{R}}
\newcommand{\set}[2]{\left\{#1\,\middle|\,#2\right\}}
\newcommand{\Tbar}{\overline{T}}
\newcommand{\Tcal}{\mathcal{T}}
\newcommand{\xbar}{\overline{x}}

\newcommand{\alphabar}{\overline{\alpha}}
\newcommand{\alphahat}{\widehat{\alpha}}
\newcommand{\kappabar}{\overline{\kappa}}
\renewcommand{\equiv}{:=}

\title{$\alpha$-Firmly Nonexpansive Operators on Metric Spaces}
\author{
{Arian B\"erd\"ellima}
\thanks{Institute for Numerical and Applied Mathematics, University of G\"ottingen.
AB was supported by the Deutscher Akademischer Austauschdienst (DAAD).
\texttt{arian.berdellima@mathematik.uni-goettingen.de}},  
\and 
{Florian Lauster}
 \thanks{Institute for Numerical and Applied Mathematics, 
University of G\"ottingen. FL was supported by 
the Deutsche Forschungsgemeinschaft (DFG, German Research Foundation) – 
Project-ID LU 1702/1-1.  
\texttt{f.lauster@math.uni-goettingen.de}}, 
\and
{D. Russell Luke} 
\thanks{Institute for Numerical and Applied Mathematics,
    University of Goettingen,
    37083 Goettingen, Germany. DRL was supported in part by 
    the Deutsche Forschungsgemeinschaft (DFG, German Research Foundation) – Project-ID 432680300 – SFB 1456.
	\texttt{r.luke@math.uni-goettingen.de}}
}
\date{\today}

\begin{document}
 \maketitle

 \begin{abstract}
We extend to $p$-uniformly convex spaces tools from the analysis of fixed point iterations 
in linear spaces.  This study is restricted to an appropriate generalization of single-valued, pointwise
$\alpha$-averaged mappings. Our main contribution is establishing a
calculus for these mappings in p-uniformly convex spaces, showing in particular how the property
is preserved under compositions and convex combinations. This is of central importance to
splitting algorithms that are built by such convex combinations and compositions, and reduces
the convergence analysis to simply verifying $\alpha$-firm nonexpansiveness of the individual
components at fixed points of the splitting algorithms. Our convergence analysis differs from what can be
found in the previous literature in that only $\alpha$-firm nonexpansiveness with respect to fixed points 
is required. Indeed we show that, if the fixed point mapping is 
pointwise  nonexpansive at all cluster points, then these 
cluster points are in fact fixed points, and convergence of the sequence 
follows.  Additionally, we provide a quantitative convergence analysis built 
on the notion of gauge metric subregularity, which we show is  
{\em necessary} for quantifiable convergence estimates.  
This allows one for the first time to prove convergence of a tremendous variety of
splitting algorithms in spaces with curvature bounded from above.
\end{abstract}

{\small \noindent {\bfseries 2010 Mathematics Subject Classification:}
  Primary 
  47H09, 
 47H10, 
 53C22,  
 53C21   

 Secondary 
 53C23,  
 53C20,  
 49M27
  }


\section{Introduction}
Our focus is on the extension  to $p$-uniformly convex spaces of tools from the analysis 
of fixed point iterations in linear spaces.  
We are indebted to the works of Kuwae \cite{Kuwae2014} and Ariza-Ruiz, Leu\c{s}tean, López-Acedo, and 
Nicolae \cite{AriLeuLop14, RuiLopNic15} who 
studied firm nonexpansiveness in nonlinear spaces, 
though the asymptotic behavior of averaged mappings in uniformly convex Banach spaces
was already studied by Baillon, Bruck and Reich in \cite{BaiBruRei78}. Reich and Shafrir
established an approach to the study of convex combinations of nonexpansive mappings in 
{\em hyperbolic spaces} \cite{ReiSha90}, the foundations for which were developed in 
\cite{GoeRei84}.  
Building on this, we follow the framework  for nonconvex optimization established
in \cite{LukTamTha18} 
which is predicated on only two fundamental elements in a 
Euclidean setting:  pointwise almost $\alpha$-averaging \cite[Definition 2.2]{LukTamTha18}
and metric subregularity \cite[Definition 2.1b]{Ioffe11}.  Almost averaged mappings 
are, in general, set-valued.  In nonlinear metric spaces, there are several difficulties 
that arise:  first, there is no straight-forward generalization of the averaging property 
since addition is not defined on general metric spaces; and second, multivaluedness, 
which comes with allowing mappings to be expansive.  The issue of multivaluedness 
introduces technical overhead, but does not, at this early stage, seem to present any 
conceptual difficulties.  The issue of violations of averagedness and 
nonexpansiveness is more fundamental.  We show that  such 
violations are unavoidable if one wants to work with resolvents.  The foundations for 
working with these difficulties are established here, but we postpone until later 
a direct study of resolvents on spaces with curvature bounded from above.

We therefore restrict our attention 
to an appropriate generalization of single-valued, pointwise $\alpha$-averaged mappings. 
This generalization leads to a definition of firm nonexpansiveness that is less retrictive than 
notions with the same name studied in \cite{GoeRei84, ReiSha87, ReiSha90, Bacak14, RuiLopNic15}, 
though, we show that our notion is implied by the previously studied objects.  
Our main contribution is establishing a calculus for these mappings in $p$-uniformly 
convex spaces, 
showing in particular how the property is preserved under compositions and convex combinations. 
This is of central importance to splitting algorithms that are built by such convex combinations and 
compositions, and reduces the convergence analysis to simply verifying quasi $\alpha$-firm 
nonexpansiveness of the individual components of the splitting algorithms.  Our convergence 
analysis also differs from what can be found in the previous literature in that only quasi $\alpha$-firm
nonexpansiveness is required.  Indeed we show (Theorem \ref{t:ne at asymp centers})  that, if the fixed point 
mapping is pointwise  
nonexpansive at the asymptotic centers of all subsequences, then all asymptotic centers 
are fixed points and weak (precisely, $\Delta$-) convergence of the 
fixed point sequence is guaranteed.  
Additionally, we provide a quantitative convergence analysis built 
on the notion of gauge metric subregularity, which we show is in fact 
{\em necessary} for quantifiable convergence estimates.  
This allows one to prove convergence of a tremendous variety of splitting algorithms 
for the first time  in spaces with curvature bounded from above.

After introducing notation, we begin 
in Section \ref{s:Metric} with the central 
object of this study given in Definition \ref{d:pafne}.  Section \ref{s:properties main} 
is devoted to developing elementary properties and 
the calculus of $\alpha$-firmly nonexpansive mappings.    
Proposition \ref{t:properties pafne} and Theorem \ref{t:RLN THm 3.1}  in   
Section \ref{s:elementary properties} 
establish asymptotic regularity.  The calculus of nonexpansive
mappings in various settings is established in Theorem \ref{t:nonexp} of  
Section \ref{s:nonexpansive}.  The calculus of quasi $\alpha$-firmly nonexpansive 
mappings is established in Theorem \ref{t:compositionthm} (compositions) of 
Section \ref{s:compositions}
and Theorem \ref{t:cvx comb pafne} (convex combinations) of Section \ref{s:convex combinations}.  
Convergence of 
fixed point iterations of $\alpha$-firmly nonexpansive mappings is studied in 
Section \ref{s:convergence} where convergence without rates is established under the 
assumption only of pointwise nonexpansiveness at the asymptotic centers of all 
subsequences
(Theorem \ref{t:ne at asymp centers}) 
and quantitative convergence in 
Theorem \ref{t:msr convergence} under the additional 
assumption of (gauge) metric subregularity (Definition \ref{d:(str)metric (sub)reg}).  
We show in Theorem \ref{t:msr necessary} that metric subregularity with some gauge
is in fact necessary to guarantee quantitative convergence estimates.  
Some basic applications and examples are presented
in Section \ref{s:Examples}. 

\section{Notation and Foundations}\label{s:notation}
Throughout, $(G, d)$ denotes a metric space.  A geodesic path emanating from a point 
$x\in G$ and extending to the point $y\in G$ is a mapping $\mymap{\gamma}{[0,l]}{G}$ 
with $\gamma(0)=x$,  $\gamma(l)=y$ and $d(\gamma(t_1), \gamma(t_2)) = |t_1-t_2|$
whenever $t_1,t_2\in [0,l]$.  When there is only one geodesic path joining any two 
points $x$ and $y$, we use the notation $z=(1-t)x\oplus ty$ where $t=d(z,x)/d(x,y)$
to denote the point on the geodesic connecting $x$ and $y$ 
such that $d(z,x)=td(x,y)$.  A geodesic space is 
a metric space $(G, d)$ for which every pair of points in $G$ is joined by 
a geodesic.  If each pair of points is joined by one and only one geodesic, the 
metric space is uniquely geodesic.  A convex set $C\subset G$ is a set containing all 
geodesics joining any two points in $C$.   
Following \cite{RuiLopNic15}
we focus on {\em $p$-uniformly convex spaces with parameter $c$} \cite{NaoSil11}:  
for $p\in (1,\infty)$,  a metric space $(G, d)$ is $p$-uniformly convex with constant 
$c>0$  whenever it is a geodesic space, and 
\begin{equation}\label{e:p-ucvx}
(\forall t\in [0,1])(\forall x,y,z\in G) \quad
d(z, (1-t)x\oplus ty)^p\leq (1-t)d(z,x)^p+td(z,y)^p - \tfrac{c}{2}t(1-t)d(x,y)^p.
\end{equation}
Examples of $p$-uniformly convex spaces include $L_p$ spaces, 
and CAT spaces (see Alexandrov \cite{Alexandrov} 
and Gromov \cite{Gromov}).  CAT(0) spaces 
can be defined by \eqref{e:p-ucvx}
with $p=2$ and $c=2$.   CAT($\kappa$) 
spaces for $\kappa>0$  are relevant for the study of phase retrieval and 
source localization \cite{LukSabTeb19}.  When the diameter of the space, 
$\diam{G}$ is bounded above by $\pi/(2\sqrt{\kappa})$, then the corresponding
CAT($\kappa$) space is $2$-uniformly convex with constant 
$c=\paren{\pi-2\sqrt{\kappa}\epsilon}\tan(\epsilon\sqrt{\kappa})$ for 
$\epsilon\in \paren{0, \pi/(2\sqrt{\kappa})}$
(see \cite{Ohta07}).  Kuwae has established bounds for the constants $p$ and $c$, 
illustrating their interdependence \cite[Proposition 2.5]{Kuwae2014}.  In particular,
we note that if $c=2$, then $p=2$.  For all other $p\in(1,+\infty)$ the constant $c$
lies in the open interval $(0,2)$. 
There is a connection between the modulus of convexity of a Banach space 
$(Y, \Vert \cdot \Vert)$ given by 
$\delta(\epsilon)\coloneqq \inf \left \{ 1-\left \Vert \frac{x+y}{2} \right\Vert%
\mid x,y\in Y, \Vert x \Vert = \Vert y \Vert =1, \Vert x-y \Vert \geq \epsilon   \right \}$, 
i.e. $Y$ will be p-uniformly convex with constant $c> 0$ if $\delta(\epsilon)\geq c \epsilon^p$ 
\cite[Remark 2.7]{Kuwae2014}.
Finally, we will use the notation $(H,d)$ to indicate a Hadamard space -  a 
{\em complete} $CAT(0)$ space -  and $\Hcal$ will indicate a Hilbert space. 

For a set $D\subseteq G$ we denote by $\overline{\co{D}}$ the closure of the convex hull of $D$. 
We denote the interior  of $D$ by $\intr{D}$.  The distance of a point $x$ to a set $D$ is 
with respect to the metric $d$ is 
denoted  $d(x, D)\equiv\inf_{z\in D}d(x,z)$ and when this distance is attained at some
point $\xbar\in D$ we call this point a {\em projection} of $x$ onto $D$.  The
mapping of a point $x$ to its set of projections is called the {\em projector}
and is denoted $P_D(x)\equiv \set{\xbar}{d(x,\xbar)=d(x,D)}$.  

A standard approach to showing the convergence of fixed point sequences is to 
show that the residual of the fixed point operator vanishes.  More precisely, 
a self-mapping $T:G\to G$ is {\em asymptotically regular} at a point $x\in G$
whenever $\lim_{k\to\infty}d(T^{(k)}x, T^{(k+1)}x)=0$.  The mapping is said to 
be asymptotically regular on $D\subset G$ if it is asymptotically regular at all 
points on $D$.  A sequence $(x_k)_{k\in\Nbb}$ is said to be asymptotically 
regular whenever $\lim_{k\to\infty}d(x_{k}, x_{k+1})=0$.  
In a Banach space setting,  firmly nonexpansive
mappings possessing fixed points are asymptotically regular, and sequences
of fixed point iterations converge weakly to a fixed point \cite{ReiSha87}.  
This is also true on $p$-uniformly convex (nonlinear) spaces 
\cite{RuiLopNic15}.  We extend these results to the generalization of averaged mappings, 
what we call $\alpha$-firmly nonexpansive mappings,  that possess fixed points in 
Theorem \ref{t:ne at asymp centers}.  Rates of convergence of the iterates $x_k$ are 
achieved in Theorem \ref{t:msr convergence} under the additional assumption 
that the fixed point mapping admits an {\em error bound}.  The notion of  
$\alpha$-firm nonexpansive operators greatly simplifies the analysis of algorithms, and opens 
the door to a study of {\em expansive} operators \cite{LukTamTha18} where 
convexity/monotonicity plays no role. 

\subsection{$\alpha$-Firmly Nonexpansive Operators in Uniformly Convex Spaces}\label{s:Metric}
Extending Bruck's origional definition of firmly nonexpansive mappings in 
uniformly convex Banach spaces \cite{Bruck73},
Ariza-Ruiz, Leu\c{s}tean and L\'opez-Acedo \cite{AriLeuLop14}
defined $\lambda$-firmly nonexpansive operators on subsets $D$ of 
W-hyperbolic spaces, as those operators satisfying 
\begin{equation}\label{e:RLN fne}
 d(Tx, Ty)\leq d((1-\lambda)x\oplus \lambda Tx, (1-\lambda)y\oplus \lambda Ty) 
 \quad \forall x,y\in D
\end{equation}
for some $\lambda\in(0,1)$.  If \eqref{e:RLN fne} holds for all $\lambda\in(0,1)$ 
the mapping $T$ is called firmly nonexpansive in \cite{AriLeuLop14, RuiLopNic15}.  
The analog to $\alpha$-averaged mappings is problematic since it 
requires the extension of geodesics beyond 
the point $Tx$ (i.e. $\lambda\in \paren{0,\tfrac{1-\alpha}{\alpha}}$).

Another notion of firm nonexpansiveness in the context of Hadamard spaces that 
is equivalent to \eqref{e:RLN fne} for an operator $T:H\to H$ and $x,y\in H$ 
uses 
\begin{equation}\label{eq:phi}
 \phi_T(t):=d((1-t)x\oplus tTx,(1-t)y\oplus tTy),\quad \mbox{for }t\in[0,1].
\end{equation}  
In \cite[Chapter 24]{GoeRei84} an operator $T:H\to H$ is called firmly nonexpansive 
whenever $\phi_T$ is nonincreasing on $[0,1]$ (see also \cite[Definition 2.1.13]{Bacak14}).

For reasons that will become apparent in Section \ref{s:convergence rate} we define 
firmness of $T$ in terms of an auxilliary function that accounts for how 
$T$ deforms the parallelogram with corners at $x$, $y$, $Tx$ and $Ty$.   
Define 
\begin{equation}\label{eq:psi}
\quad \psi^{(p,c)}_T(x,y) \coloneqq 
\tfrac{c}{2}\paren{d(Tx, x)^p+d(Ty, y)^p + d(Tx, Ty)^p + d(x, y)^p  - d(Tx, y)^p   - d(x,Ty)^p }.
\end{equation}
In a Hilbert space setting this is recognizable as 
\[
 \|(\Id-T)x-(\Id-T)y\|^2 =  \psi^{(2,2)}_T(x,y).
\]
The next definition generalizes 
firmness to mappings that may violate the inequality defining firmness in a manner 
analogous to such mappings studied in \cite{LukTamTha18, LukMar20}.  
We do not fully develop the potential of this extension here, but will use it 
in a result about proximal mappings in Corollary \ref{t:prox aafne}. 
\begin{mydef}\label{d:pafne}
 Let $(G, d)$ be a $p$-uniformly convex metric space with constant $c$.  
 The operator $T:G\to G$ 
 is pointwise almost $\alpha$-firmly nonexpansive at $y\in D\subset G$ on $D$ 
 if  
\begin{equation}
\label{eq:pafne}
\exists \alpha\in(0,1), \epsilon>=0:\quad  d(Tx,Ty)^p\leq (1+\epsilon)d(x,y)^p-\tfrac{1-\alpha}{\alpha}\psi^{(p,c)}_T(x,y)\quad 
\forall x\in D.
\end{equation}
The smallest $\epsilon$ 
for which \eqref{eq:pafne} holds is called the {\em violation}.  If \eqref{eq:pafne} holds with $\epsilon=0$, then 
$T$ is pointwise $\alpha$-firmly nonexpansive at $y\in D\subset G$ on $D$. 
If \eqref{eq:pafne} holds at all $y\in D$ with the same constant $\alpha$, 
then $T$ is said to be (almost) $\alpha$-firmly nonexpansive on $D$.  If $D=G$ the mapping 
$T$ is simply said to be (almost) $\alpha$-firmly nonexpansive.  If $D\supset \Fix T\neq\emptyset$ and 
\eqref{eq:pafne} holds at all $y\in \Fix T$ with the same constant $\alpha$ then $T$ is 
said to be almost quasi $\alpha$-firmly nonexpansive.  
\end{mydef}

The transport discrepancy $\psi_T^{(p,c)}$ is closely related to another object 
used by  Berg and Nikolaev \cite{Berg} in the study of CAT(0) space $(G,d)$. 
In $p$-uniformly convex spaces with constant $c$ this takes the form:
\begin{equation}
\label{eq:4point p-uniform}
\Delta^{(p,c)}(x,y,u,v):=\frac{c}{4}(d(x,v)^p+d(y,u)^p-d(x,u)^p-d(y,v)^p).
\end{equation}
Specializing to a Hilbert space $\mathcal{H}$, this is identifiable with 
the inner product:
\begin{equation}
\label{eq:4point}
\langle x-y,u-v\rangle=\frac{1}{2}(\|x-v\|^2+\|y-u\|^2-\|x-u\|^2-\|y-v\|^2)
 = \Delta^{(2,2)}(x,y,u,v).
\end{equation}
In the context of characterizing the regularity of a mapping $T$, $u=Tx$ and $v=Ty$
it is convenient to denote  $\Delta^{(p,c)}_T(x,y):=\Delta^{(p,c)}(x,y,Tx,Ty)$. 
This object was introduced in \cite[Chapter 7]{Berdellima20} in the context of Hadamard spaces ($p=c=2$)
where it was called the discrepancy mapping. 
In particular, note that 
\begin{equation}\label{e:psi-delta}
	\psi_T^{(p,c)}(x,y) =  \tfrac{c}{2}\paren{d(Tx, Ty)^p+d(x, y)^p}	 - 2\Delta^{(p,c)}_T(x,y).
\end{equation}
This leads to the following equivalent characterization of pointwise $\alpha$-firm nonexpansiveness. 
\begin{proposition}\label{t:pafne Delta}
Let  $(G, d)$ be a p-uniformly convex space with constant $c>0$ and   
let $T:D\to G$ for $D\subset G$.  The mapping $T$  
    is pointwise almost $\alpha$-firmly nonexpansive at $y\in D$ 
    with constant $\alpha$ and violation $\epsilon$ on $D$ 
    if and only if 
\begin{equation}
\label{eq:pafne Delta}
\paren{\alpha+(1-\alpha)\tfrac{c}{2}}d(Tx,Ty)^p+
\paren{\alpha(1+\epsilon)-(1-\alpha)\tfrac{c}{2}}d(x,y)^p
\leq  2(1-\alpha)\Delta^{(p,c)}_T(x,y),\quad\forall x\in D.
\end{equation}
\end{proposition}
\begin{proof}
Starting with the definition of pointwise $\alpha$-firm nonexpansiveness at 
 $y\in D$ we have for all $x\in D$
 \begin{eqnarray*}
	 &d(Tx, Ty)^p\leq d(x,y)^p - \tfrac{1-\alpha}{\alpha}\psi^{(p,c)}_T(x,y) &  \\
	 &\iff&\\
	 &(1+\epsilon)\paren{\alpha+(1-\alpha)\tfrac{c}{2}}d(Tx, Ty)^p\leq 
	 \paren{\alpha(1+\epsilon)-(1-\alpha)\tfrac{c}{2}}d(x,y)^p\qquad\qquad\qquad\qquad&\\
	 &\qquad\qquad\qquad\qquad\qquad\qquad
	 - (1-\alpha)\tfrac{c}{2}\paren{d(Tx,x)^p+ d(Ty,y)^p - d(Tx,y)^p - d(x,Ty)^p}.&\\
 \end{eqnarray*}
Noting that 
\[
 	 - (1-\alpha)\tfrac{c}{2}\paren{d(Tx,x)^p+ d(Ty,y)^p - d(Tx,y)^p - d(x,Ty)^p}
	 = (1-\alpha)2\Delta^{(p,c)}_T(x,y)
\]
establishes the equivalence.  
\end{proof}

When $p=2$ and $c=2$, we show below  in Proposition \ref{t:properties pafne}\eqref{t:properties pafne vi}
that mappings $T$ 
satisfying \eqref{e:RLN fne} also satisfy 
\begin{equation}
\label{eq:fneH}
d(Tx,Ty)^p\leq  \Delta^{(p,c)}_T(x,y)\qquad \forall x,y\in G.
\end{equation}
Indeed, in this case,   
\eqref{eq:fneH} is equivalent to Property $P_2$ of \cite{RuiLopNic15}, which 
Ariza-Ruiz, López-Acedo, and Nicolae 
show holds for mappings satisfying \eqref{e:RLN fne}.  
When \eqref{eq:fneH} holds only pointwise at $y$ on a neighborhood $D\subset G$ of $y$ 
we write 
\begin{equation}
\label{eq:pfneH}
d(Tx,Ty)^p\leq  \Delta^{(p,c)}_T(x,y)\qquad \forall x\in D.
\end{equation}
This provides for a natural extension of monotonicity:  an operator $T:G\to G$ is monotone whenever
\begin{equation}
\label{eq:monotoneH}
\Delta^{(p,c)}_T(x,y)\geqslant 0, \qquad\forall x,y\in G.
\end{equation}
From these definitions it 
follows that if $T$ satisfies \eqref{e:RLN fne} for all $\lambda\in(0,1)$ then it is monotone. 
The correspondence between firmly nonexpansive and nonexpansive mappings, 
in contrast, is a consequence of the metric equivalent to the Cauchy-Schwarz inequality
and does not hold in general metric spaces.  Nevertheless, the correspondence is recovered
for $p$-uniformly convex spaces for {\em pointwise} firmly nonexpansive mappings
at their fixed points.  
The first result below sorts out these various notions of firm nonexpansiveness on 
CAT(0) spaces. 

\begin{proposition}\label{t:Bacak fne}
Let $(G,d)$ be a CAT(0) space.  A mapping $T:G\to G$ satisfies \eqref{e:RLN fne} for all 
$\lambda\in[0,1]$ if and only if  
 $\phi_T(t)$ defined by \eqref{eq:phi} is a nonincreasing function on $[0,1]$ for all 
 $x,y\in G$.   Any mapping $T$ satisfying \eqref{e:RLN fne} for all 
$\lambda\in(0,1]$ is $\alpha$-firmly 
nonexpansive in the sense of \eqref{eq:pafne} with $\alpha=1/2$.
\end{proposition}
\begin{proof}
Equivalence of \eqref{e:RLN fne} for all 
$\lambda\in(0,1)$ and \eqref{eq:phi} is immediate from the definition.

To prove the rest of the theorem, note that assumption $\phi_T(t)$ is a nonincreasing function 
on $[0,1]$ for all $x,y\in G$ implies that 
$\phi_T(1)\leq  \phi_T(t)$ for all $t\in[0,1]$. On the other hand from applying 
\eqref{e:p-ucvx} with $p=2$ and $c=2$
twice we obtain 
\begin{align*}
\phi^2_T(t)
&\leq (1-t)^2d(x,y)^2+t^2d(Tx,Ty)^2\\
&+t(1-t)[d(x,Ty)^2+d(y,Tx)^2-d(x,Tx)^2-d(y,Ty)^2]
\end{align*}
Hence 
$$\phi^2_T(1)=d(Tx,Ty)^2\leq 
(1-t)^2d(x,y)^2+t^2d(Tx,Ty)^2+2t(1-t)\Delta^{(2,2)}_T(x,y)$$
equivalently 
\begin{equation}\label{eq:phi ineq}
(1-t^2)d(Tx,Ty)^2\leq (1-t)^2d(x,y)^2+2t(1-t)\Delta^{(2,2)}_T(x,y) 
\end{equation}
Dividing by $1-t$ and letting $t\uparrow 1$ yields $d(Tx,Ty)^2\leq \Delta^{(2,2)}_T(x,y)$.
By Proposition \ref{t:pafne Delta} this is just $\alpha$-firm
nonexpansiveness with $\alpha=1/2$ as claimed. 
\end{proof}

\section{Properties of Pointwise Nonexpansive and $\alpha$-Firmly Nonexpansive Mappings
in Metric Spaces}\label{s:properties main}
Before developing the calculus of $\alpha$-firmly nonexpansive mappings, 
we begin with some elementary properties of pointwise $\alpha$-firmly nonexpansive mappings. 

\subsection{Elementary Properties of $\alpha$-Firmly Nonexpansive Operators} 
\label{s:elementary properties}
\begin{proposition}\label{t:properties pafne}
Let  $(G, d)$ be a p-uniformly convex space with constant $c>0$ and   
let $T:D\to G$ for $D\subset G$.
\begin{enumerate}[(i)]
\item\label{t:properties pafne ii}
\begin{equation}\label{e:psi-Fix T}
\psi_T^{(p,c)}(x, y)=\tfrac{c}{2}d(Tx, x)^p \quad
\mbox{whenever } y\in\Fix T.
\end{equation}
For fixed $y\in\Fix T$ the function $\psi_T^{(p,c)}(x,y)\geq 0$ for all $x \in D$ and  $\psi_T^{(p,c)}(x,y)= 0$ 
only when  $x\in\Fix T$.
\item\label{t:properties pafne iii}  Let $y\in \Fix T$.  $T$ is  pointwise $\alpha$-firmly nonexpansive at $y$ 
on $D$ if and only if  
\begin{equation}\label{e:P1}
\exists \alpha\in[0,1):\quad 	d(Tx,y)^p\leq d(x,y)^p - \tfrac{1-\alpha}{\alpha}\tfrac{c}{2}d(Tx,x)^p\quad
	\forall x\in D.
\end{equation}
In particular, $T$ is quasi $\alpha$-firmly nonexpansive on $D$ whenever $T$ possesses fixed points and 
\eqref{e:P1} holds at all $y\in \Fix T$ with the same constant $\alpha\in [0,1)$. 
\item\label{t:properties pafne iv}  If $T$ is pointwise $\alpha$-firmly nonexpansive at $y\in\Fix T$ on $D$ with 
constant $\underline\alpha\in[0,1)$ then it is pointwise $\alpha$-firmly nonexpansive at $y$ on $D$ for all 
$\alpha\in[\underline\alpha,1]$.   In particular, if $T$ is pointwise $\alpha$-firmly nonexpansive at $y\in\Fix T$ on D, 
then it is pointwise nonexpansive at $y$ on D.  
\item\label{t:properties pafne v} If at $y\in \Fix T$  
\begin{equation}\label{e:RLN fne2}
\exists\lambda\in(0,1):\quad	 d(Tx, y)\leq d((1-\lambda)x\oplus \lambda Tx, y) 
 \quad \forall x\in D, 
\end{equation}
then $T$ is pointwise $\alpha$-firmly nonexpansive  at $y$ with constant 
$\alpha=1/(\lambda+1)$ on $D$. 
\item\label{t:properties pafne vi}  Let $p=2$ and $c=2$ (that is, $(G, d)$ is a CAT(0) space).
Then
\begin{enumerate}
	\item $\Delta_T^{(2,2)}(x,y)\leq d(x, y)d(Tx, Ty)$  for all $x,y\in G$;
	\item  $\psi_T^{(2,2)}(x,y)\geq 0$ 
        for all $x,y\in G$;
	\item the following are equivalent: 
	\begin{enumerate}[(1)]
        \item $T$ is pointwise $\alpha$-firmly nonexpansive at $y$ with constant $\underline{\alpha}$ on 
		$D$ 
		\item $T$ is pointwise $\alpha$-firmly nonexpansive at $y$ for all constants 
		$\alpha\in[\underline{\alpha},1]$ on $D$
		\item 
            \begin{equation*}
                (1+\lambda)d(Tx, Ty)^2\leq (1-\lambda)d(x,y)^2+
                2\lambda\Delta_T^{(2,2)}(x,y)
                \quad\forall x\in D, \forall \lambda\in[0,\tfrac{1-\underline\alpha}{\underline\alpha}];
            \end{equation*}
      \end{enumerate}
    \item $T$ satisfying \eqref{e:RLN fne} is $\alpha$-firmly nonexpansive 
		with constant $\alpha=\tfrac{1}{1+\lambda}$ on $D$.  
	\end{enumerate}
\end{enumerate}
\end{proposition}
\begin{proof}
\eqref{t:properties pafne ii}.  Fix $y\in \Fix T$. Equation \eqref{e:psi-Fix T} follows directly from \eqref{eq:psi},
from which the rest of the claim is immediate. 

\eqref{t:properties pafne iii}.  This is immediate from the definition and part \eqref{t:properties pafne ii}.

\eqref{t:properties pafne iv}.  This follows immediately from parts \eqref{t:properties pafne ii} and \eqref{t:properties pafne iii}.

\eqref{t:properties pafne v} Starting with \eqref{e:RLN fne2}, by the characterization of 
$p$-uniformly convex spaces \eqref{e:p-ucvx}
\begin{eqnarray*}
	d(Tx,y)^p&\leq& d((1-\lambda)x\oplus \lambda Tx,y)^p\\
&\leq& 	 (1-\lambda)d(x,y)^p+\lambda d(Tx,y)^p - (1-\lambda)\lambda\tfrac{c}{2}d(Tx,x)^p 
\end{eqnarray*}
for all $x\in D$ and some $\lambda\in (0,1)$.  Rearranging terms yields
\[
\exists\lambda\in(0,1):\quad 
d(Tx,y)^p\leq d(x,y)^p - \lambda\tfrac{c}{2}d(Tx,x)^p\quad\forall x\in D.
\]
When $y\in\Fix T$, by part \eqref{t:properties pafne iii}, this is 
equivalent to $T$ being pointwise $\alpha$-firmly nonexpansive at $y$ with constant 
$\alpha=\tfrac{1}{\lambda+1}$ on $D$.  

\eqref{t:properties pafne vi}(a) This is a direct consequence of the 
inequality 
\begin{equation}
\label{eq:CS}
\Delta^{(2,2)}(x,y,u,v)\leq  d(x,y)d(u,v)
\end{equation}
(see \cite[Theorem 2.3.1]{Jost97} or 
\cite[Lemma 2.1]{LanPavSch00}).

\eqref{t:properties pafne vi}(b) By \eqref{e:psi-delta} and part \eqref{t:properties pafne vi}(a)
\begin{eqnarray*}
	\psi^{(2,2)}_T(x,y) &=& d(x, y)^2 + d(Tx, Ty)^2 - 2\Delta^{(2,2)}_T(x,y)\\
	&\geq& d(x, y)^2 + d(Tx, Ty)^2 - 2 d(x,y)d(Tx,Ty)\\
	&=& \paren{d(x, y) - d(Tx, Ty) }^2\geq 0
\end{eqnarray*}
for all $x,y\in G$, as claimed.  

\eqref{t:properties pafne vi}(c)  By part \eqref{t:properties pafne vi}(b) 
if $T$ is pointwise $\alpha$-firmly nonexpansive at $y$ with constant $\underline\alpha$ on $D$ 
then 
\begin{eqnarray*}
 &
 d(Tx, Ty)^2\leq d(x,y)^2 - \tfrac{1-\underline\alpha}{\underline\alpha}\psi^{(2,2)}_T(x,y)\quad\forall x\in D&\\
&\underset{(2)}{\overset{{(1)}}{\iff}}&\\
 &
 d(Tx, Ty)^2\leq d(x,y)^2 - \tfrac{1-\alpha}{\alpha}\psi^{(2,2)}_T(x,y)\quad\forall x\in D, 
 \forall \alpha\in[\underline\alpha,1]&\\
   &\iff&\\
&
d(Tx, Ty)^2\leq d(x,y)^2 - \lambda\psi^{(2,2)}_T(x,y)\quad\forall x\in D, 
 \forall\lambda\in[0,\tfrac{1-\underline\alpha}{\underline\alpha}]&\\
&\underset{(3)}{\overset{{(2)}}{\iff}}&\\
  &
  (1+\lambda)d(Tx, Ty)^2\leq (1-\lambda)d(x,y)^2 +2 \lambda\Delta^{(2,2)}_T(x,y)\quad\forall x\in D, 
 \forall\lambda\in[0,\tfrac{1-\underline\alpha}{\underline\alpha}]&\\
\end{eqnarray*}
where the last equivalence follows from \eqref{e:psi-delta}.

\eqref{t:properties pafne vi}(d)
For fixed $\lambda\in[0,1)$ \eqref{eq:phi ineq} in the proof of Proposition \ref{t:Bacak fne}  yields 
%
%
\begin{eqnarray*}
(1+\lambda)d(Tx, Ty)^2&\leq& (1-\lambda)d(x,y)^2  
+2\lambda \Delta^{(2,2)}_T(x,y), \quad\forall x,y\in D.
\end{eqnarray*}
By \eqref{t:properties pafne vi}(c), this implies that 
$T$ is $\alpha$-firmly nonexpansive at all $y\in D$ for any constant 
$\alpha\in[\tfrac{1}{1+\lambda}, 1]$ on $D$.
This completes the proof. 
\end{proof}
\begin{remark}\label{r:P1}
	Property \eqref{e:P1} is a specialization of Property (P1) of \cite{RuiLopNic15} to 
	$\alpha$-firmly nonexpansive mappings (as we define them) on $p$-uniformly convex spaces.  
\end{remark}

Closedness and convexity of the set of fixed points of nonexpansive mappings is easily 
established.  Note, however, that convexity of the 
fixed point set depends on convexity of the domain.  In Section \ref{s:convergence rate}
we will not require convexity of the domain.  
\begin{lemma}
 \label{l:fixpointcvx}
 Let $(G,d)$ be a $p$-uniformly convex metric space with constant $c>0$ and let $D\subseteq G$ 
 be closed and convex.
 Let $T:G\to G$ be pointwise nonexpansive at all $y\in\Fix T\cap D\neq\emptyset$ on $D$. Then $\Fix T\cap D$ 
 is a closed and convex set.
\end{lemma}
\begin{proof}
This statement for $T$ a nonexpansive (not pointwise) mapping on a uniquely geodesic space is in 
\cite[Lemma 6.2]{AriLeuLop14}.  Their proof also works for pointwise nonexpansive mappings.
\end{proof}

In \cite{RuiLopNic15} the central property of {\em asymptotic regularity} of a mapping $T$
at its fixed points hinges on (i) existence of fixed points, and (ii) the validity of 
inequality \eqref{e:P1} at all $y\in\Fix T$.  
Proposition \ref{t:properties pafne}\eqref{t:properties pafne iii} 
shows that these two requirements are {\em equivalent} to $T$ being quasi $\alpha$-firmly nonexpansive.  
We show in Theorem \ref{t:ne at asymp centers} that, as a consequence of the next theorem on asymptotic regularity,  
pointwise $\alpha$-firm nonexpansiveness at reasonable subsets of fixed points is 
all that is needed to achieve weak convergence of fixed point iterations.
\begin{thm}\label{t:RLN THm 3.1}
Let $(G, d)$ be a $p$-uniformly convex space, let $D\subset G$, and let  
$T:G\to G$ with $\Fix T\cap D$ nonempty and $T(D)\subseteq D$.  
Suppose further that $T$ is pointwise $\alpha$-firmly 
nonexpansive at all $y\in\Fix T\cap D$ on $D$.  
Then given any starting point $x_0\in D$ the sequence 
$(x_k)_{k\in\Nbb}$ defined by $x_{k+1}=Tx_k$ is asymptotically regular on $D$.  
\end{thm}
\begin{proof}
	By Proposition \ref{t:properties pafne}\eqref{t:properties pafne iii} and Remark \ref{r:P1},
	the statement is a specialization of \cite[Theorem 3.1]{RuiLopNic15} to the case of 
	just a single operator.  
\end{proof}
We show below, that compositions and convex combinations of (quasi) $\alpha$-firmly 
nonexpansive mappings are quasi $\alpha$-firmly nonepxansive.  Therefore, by 
the theorem above, fixed point iterations 
of such compositions and convex combinations are asymptotically regular.

\subsection{Calculus of Nonexpansive Operators}\label{s:nonexpansive}
Nonexpansiveness is preserved under compositions and, with some restrictions, under convex combinations, 
as the next result shows.  
\begin{thm}
\label{t:nonexp}
Let $D\subset G$ where $(G, d)$ is a $p$-uniformly convex space with constant $c>0$ and 
let $T_1, T_2: D\to G$.
 \begin{enumerate}[(i)]
\item\label{t:nonexp0i} If $\Fix T_2\cap \Fix T_1\neq\emptyset$ then  
 any convex combination of $T_2$ and $T_1$ is 
pointwise nonexpansive at $y\in \Fix T_2\cap \Fix T_1$ on $D$ whenever $T_2$ and $T_1$ 
are pointwise nonexpansive there.
 \item\label{t:nonexpi} When $(G, d)$ is a CAT(0) space, (that is, $p=2$, and $c=2$), 
 then any convex combination of $T_2$ and $T_1$ is 
pointwise nonexpansive at $y\in D$ on $D$ whenever $T_2$ and $T_1$ are pointwise nonexpansive there. 
\item\label{t:nonexpii} For $D_1\equiv \{z~|~z=T_1y, y\in D\}$ let $T_2:D_1\to G$ 
 be pointwise nonexpansive at $T_1y\in D_1$ on $D_1$ and let $T_1$ be pointwise nonexpansive 
 at $y\in D$.  Then the composition $T_2\circ T_1$ is pointwise nonexpansive  
 at $y$ on $D$.
\end{enumerate}
In particular, 
the set of all nonexpansive operators in CAT(0) spaces is closed under 
compositions and convex combinations. 
\end{thm}
\begin{proof}
	\eqref{t:nonexp0i}.  Let $\lambda\in(0,1)$ and define $T_\lambda :=(1-\lambda)T_2\oplus\lambda T_1$.
	Applying \eqref{e:p-ucvx} first to $T_\lambda y$ and then to $T_\lambda x$ yields
	\begin{eqnarray}
	 d(T_\lambda x, T_\lambda y)^p&\leq& (1-\lambda)d(T_\lambda x, T_2y)^p+\lambda d(T_\lambda x, T_1y)^p
	 - \tfrac{c}{2}\lambda(1-\lambda)d(T_2y, T_1y)^p\nonumber \\
	 &\leq &
	 (1-\lambda)^2d(T_2x, T_2y)^p+\lambda^2 d(T_1x, T_1y)^p +
	 \lambda(1-\lambda)\paren{d(T_1x, T_2y)^p+d(T_2x, T_1y)^p}\nonumber\\
	 &&\qquad - \frac{\lambda(1-\lambda)c}{2}\paren{d(T_2x, T_1x)^p+ d(T_2y, T_1y)^p}.
	 \label{e:Big Thief}
	\end{eqnarray}
For $y\in\Fix T_2\cap \Fix T_1$ this yields  
	\begin{eqnarray*}
	 d(T_\lambda x, T_\lambda y)^p&\leq& 
	 (1-\lambda)d(T_2x, y)^p+\lambda d(T_1x, y)^p \\
	 &\leq&
	 (1-\lambda)d(x, y)^p+\lambda d(x, y)^p = d(x, y)^p\quad \forall x\in D,
	\end{eqnarray*}
	where the last inequality uses pointwise nonexpansiveness of $T_1$ and $T_2$ at $y$.  
Therefore $(1-\lambda)T_2\oplus\lambda T_1$ is pointwise nonexpansive at $y\in \Fix T_2\cap\Fix T_1$
on $D$ for all $\lambda\in [0,1]$, as claimed.  

\eqref{t:nonexpi}.  Let $\lambda\in(0,1)$ and define $T_\lambda :=(1-\lambda)T_2\oplus\lambda T_1$. 
Applying \eqref{e:Big Thief} with $p=2$ and $c=2$ yields
\begin{align*}
d(T_\lambda x,T_\lambda y)^2&\leq  (1-\lambda)^2d(T_2x,T_2y)^2+\lambda^2 d(T_1x,T_1y)^2\\
&+(1-\lambda)\lambda (d(T_2x,T_1y)^2+d(T_1x,T_2y)^2-d(T_2y,T_1y)^2-d(T_2x,T_1x)^2)\\
& = (1-\lambda)^2d(T_2x,T_2y)^2+\lambda^2 d(T_1x,T_1y)^2+ 
(1-\lambda)\lambda \Delta^{(2,2)}(T_2x, T_2y, T_1x, T_1y)
\end{align*}
for any $x\in D$, where $\Delta^{(2,2)}$ is defined by \eqref{eq:4point p-uniform}. 
On the other hand $(G,d)$ is a CAT(0)-space, so \eqref{eq:CS} holds, and in particular, 
$$ \Delta^{(2,2)}(T_2x, T_2y, T_1x, T_1y)\leq  2d(T_2x,T_2y)d(T_1x,T_1y)$$
for any $x,y\in G$.  
Therefore 
\[
d(T_\lambda x,T_\lambda y)^2\leq \paren{(1-\lambda)d(T_2x,T_2y)+\lambda d(T_1x,T_1y)}^2
\quad\forall x\in D.
\]
By assumption both $T_2$ and $T_1$ are pointwise nonexpansive at $y$ on $D$, so
\[
d(T_\lambda x,T_\lambda y)^2\leq ((1-\lambda)d(x,y)+\lambda d(x,y))^2=d(x,y)^2
\quad\forall x\in D
\]
and hence $d(T_\lambda x,T_\lambda y)\leq  d(x,y)$ for all $x\in D$ as claimed.

\eqref{t:nonexpii}. Let $\Tbar\equiv T_2\circ T_1$.  Then 
\begin{eqnarray*}
&d(\Tbar x,\Tbar y)=d(T_2T_1x,T_2T_1y)\leq  d(T_1x,T_1y)\quad\forall T_1x\in D_1&\\
&\iff&\\
&d(\Tbar x,\Tbar y)\leq  d(T_1x,T_1y)\quad\forall x\in D&
\end{eqnarray*}
since $T_2$ is pointwise nonexpansive at $T_1y\in D_1$ on $D_1$.  But since $T_1$ is pointwise nonexpansive 
at $y$ on $D$ this yields 
\[
 d(T_1x,T_1y)\leq d(x,y)\quad \forall x\in D
\]
which establishes the claim and completes the proof.
\end{proof}


\subsection{Compositions of $\alpha$-Firmly Nonexpansive Operators} 
\label{s:compositions}
In this section we show how  
the composition of two $\alpha$-firmly nonexpansive operators is again 
$\alpha$-firmly nonexpansive.  In general this does not hold, but the 
property does hold pointwise at fixed points of the composite operator, and 
for many applications this is all that is needed.  
The next lemma relates the fixed points of compositions of $\alpha$-firmly nonexpansive mappings
to the intersection of the fixed points of the individual mappings.   
\begin{lemma}\label{t:intersections}
Let $(G, d)$ be a metric space.   
\begin{enumerate}[(i)]
\item\label{t:intersections i} Let $T_2, T_1 :G\to G$ satisfy 
$\Fix T_2\cap \Fix T_1\neq\emptyset$.  
If $T_2$ is pointwise nonexpansive at all 
 $y\in \Fix T_2\cap \Fix T_1\neq\emptyset$ on $D\subset G$ where $\Fix T_2\cap \Fix T_1\subset D$, and 
 $T_1$ is pointwise $\alpha$-firmly nonexpansive at all 
$y\in \Fix T_2\cap \Fix T_1$ on $D$,   then 
 $\Fix T_2T_1=\Fix T_2\cap \Fix T_1$.
 \item\label{t:intersections ii} Let $\{T_1, T_2, \dots, T_m\}$ be a collection of 
 quasi $\alpha$-firmly nonexpansive mappings, each with respective 
 constants $\alpha_j$ on $D\supset \cap_{j=1}^m \Fix T_j\neq \emptyset$.
 Then $\Fix \paren{T_m\circ T_{m-1}\circ\cdots\circ T_1}=\cap_{j=1}^m \Fix T_j$.
\end{enumerate}
\end{lemma}
\begin{proof}
\eqref{t:intersections i}
 The inclusion $\Fix T_2\cap \Fix T_1\subseteq \Fix T_2T_1$ is obvious. Now let 
 $x$ be any point is $\Fix T_2T_1$ and $y$ any point in $\Fix T_2\cap \Fix T_1$. 
 There are three mutually exclusive 
 cases. First let $T_1x\in\Fix T_2$ then $T_1x=T_2T_1x=x$ implies $x\in \Fix T_2\cap \Fix T_1$. 
 Second let $x\in \Fix T_1$ then $x=T_2T_1x=T_2x$ implies $x\in \Fix T_2\cap \Fix T_1$. Finally, 
 let $x\notin \Fix T_1$ and $T_1x\notin \Fix T_2$. This yields
\begin{align*}
d(x,y)^p=d(T_2T_1x,T_2y)^p\leq  d(T_1x,y)^p&=d(T_1x,T_1y)^p\\
&\leq  d(x,y)^p-\frac{1-\alpha}{\alpha}\tfrac{c}{2}d(x,T_1x)^p
\end{align*}
where the first inequality follows from pointwise nonexpansiveness of $T_2$ at  
$y\in \Fix T_2\cap \Fix T_1$ on $D$, and the 
second inequality follows from the assumption that $T_1$ is 
pointwise $\alpha$-firmly nonexpansive at $y\in \Fix T_2\cap \Fix T_1$ on $D$ and 
Proposition \ref{t:properties pafne}\eqref{t:properties pafne iii}.  But this implies that 
$d(x,y)^p<d(x,y)^p$, which  is impossible.   
Therefore $\Fix T_2T_1\subseteq \Fix T_2\cap\Fix T_1$ as claimed.

\eqref{t:intersections ii}  In light of Remark \ref{r:P1}, this follows immediately from 
\cite[Proposition 2.1]{RuiLopNic15}.
\end{proof}

\begin{lemma}\label{t:afne of compositions}
 Let $(G, d)$ be a $p$-uniformly convex space with constant $c$ and let $D\subset G$.  
 Let  $T_1:D\to G$ be pointwise $\alpha$-firmly nonexpansive at $y$ on 
 $D$ with constant $\alpha_1$
 and let $T_2:D_1\to G$ be pointwise $\alpha$-firmly nonexpansive at $T_1y$ on 
 $D_1$ with constant 
 $\alpha_2$ where $D_1\equiv\{ T_1x~|~x\in D\}$.  
 Then the composition $\Tbar\equiv T_2\circ T_1$ is pointwise $\alpha$-firmly nonexpansive 
 at $y$ on $D$ whenever 
 \begin{equation}
\label{eq:afne composition}
\exists~ \alphabar\in(0,1):\quad 
\frac{1-\alpha_{1}}{\alpha_{1}}\psi^{(p,c)}_{T_1}(x,y)
 + \frac{1-\alpha_{2}}{\alpha_{2}}\psi^{(p,c)}_{T_2}(T_1x,T_1y)
\geq \frac{1-\alphabar}{\alphabar}\psi^{(p,c)}_{\Tbar}(x,y)
\quad \forall x\in D.
\end{equation}
\end{lemma}
\begin{proof}
Since $T_2$ is pointwise $\alpha$-firmly nonexpansive at $T_1y$ with 
constant $\alpha_2$ on $D_1$ we have 
\[ 
d(\Tbar x,\Tbar y)^p\leq  d(T_1x,T_1y)^p-\frac{1-\alpha_2}{\alpha_2}\psi^{(p,c)}_{T_2}(T_1x,T_1y)
 \quad \forall T_1x\in D_1
\]
where $\psi^{(p,c)}_{T_2}$ is defined by \eqref{eq:psi}.
On the other hand, $\{x~|~T_1x\in D_1\} = D$, and  since $T_1$ is 
pointwise $\alpha$-firmly nonexpansive at $y$ on $D$ with 
constant $\alpha_1$ we have 
\begin{align*}
d(\Tbar x,\Tbar y)^p\leq  d(x,y)^p&
-\frac{1-\alpha_1}{\alpha_1}\psi^{(p,c)}_{T_1}(x, y)
-\frac{1-\alpha_2}{\alpha_2}\psi^{(p,c)}_{T_2}(T_1x,T_1y),
\quad\forall x\in D.
\end{align*}
Whenever \eqref{eq:afne composition} holds, we can conclude that  
\begin{eqnarray*}
\exists~ \alphabar\in(0,1):\quad 
d(\Tbar x,\Tbar y)^p\leq
d(x,y)^p-
\frac{1-\alphabar}{\alphabar}\psi^{(p,c)}_{\Tbar }(x,y)
\quad\forall x\in D.
\end{eqnarray*}
\end{proof}

 \begin{thm}
 \label{t:compositionthm}
Let $(G, d)$ be a $p$-uniformly convex space.  Let  
$T_1:D\to G$ for $D\subset G$, $T_2:D_1\to G$ for $D_1\equiv\{T_1x~|~x\in D\}$, 
define $\Tbar\equiv T_2\circ T_1$ 
and let $\Fix \Tbar \subset D$ and $\Fix T_1\cap \Fix T_2$ both be nonempty.  
If $T_1$ is pointwise $\alpha$-firmly 
nonexpansive at all $y\in \Fix \Tbar $ 
 with constant $\alpha_1$ on $D$, and  
 if $T_2$ is pointwise $\alpha$-firmly nonexpansive at all $y\in\Fix \Tbar $ 
 with constant $\alpha_2$ on $D_1$,  then the composite operator 
 $\Tbar $ is quasi $\alpha$-firmly nonexpansive  on $D$ with constant 
 \begin{equation}\label{eq:alphabar}
	 \alphabar=\tfrac{\alpha_1+\alpha_2-2\alpha_1\alpha_2}%
	 {\frac{c}{2}\paren{1-\alpha_1-\alpha_2+\alpha_1\alpha_2}+\alpha_1+\alpha_2-2\alpha_1\alpha_2}.
 \end{equation}
\end{thm}
\begin{proof}
	By Lemma \ref{t:afne of compositions}, it suffices to show \eqref{eq:afne composition}
	at all points  $y\in \Fix \Tbar $.  First, note that by Lemma \ref{t:intersections},  
	$\Fix \Tbar =\Fix T_2\cap \Fix T_1$, 
	so by \eqref{e:psi-Fix T} we have  
	$\psi^{(p.c)}_{T_1}(x,y)=\tfrac{c}{2}d(x,T_1x)^p$, $ \psi^{(p.c)}_{T_2}(T_1x,T_1y)=\tfrac{c}{2}d(T_1x,\Tbar x)^p$, and 
 $\psi^{(p.c)}_{\Tbar }(x,y)=\tfrac{c}{2}d(x,\Tbar x)^p$. 
 Then whenever $y\in\Fix \Tbar $ the inequality \eqref{eq:afne composition} simplifies to 
 \begin{equation}\label{eq:afne composition Fix TS}
\exists \kappabar>0:\quad	 
\kappa_1 d(x, T_1x)^p+\kappa_2 d(T_1x, \Tbar x)^p\geq \kappabar d(x, \Tbar x)^p  
\qquad \forall x\in D,
 \end{equation}
 where $\kappa_1\equiv\frac{1-\alpha_1}{\alpha_1}$, 
 $\kappa_2\equiv\frac{1-\alpha_2}{\alpha_2}$ and 
 $\kappabar\equiv\frac{1-\alphabar}{\alphabar}$ with $\alphabar\in(0,1)$.
By \eqref{e:p-ucvx}, we have 
\begin{eqnarray}
\tfrac{c}{2}t(1-t)d(x,\Tbar x)^p&\leq& \tfrac{c}{2}t(1-t)d(x,\Tbar x)^p+ 
d(T_1x, (1-t)x\oplus t\Tbar x)^p\nonumber\\
&\leq& (1-t)d(T_1x,x)^p+
td(T_1x,\Tbar x)^p
\quad\forall x\in G, \forall t\in(0,1).
\label{e:p-ucvx2}
\end{eqnarray}
Letting $t=\tfrac{\kappa_2}{\kappa_1+\kappa_2}$ yields
$(1-t)=\tfrac{\kappa_1}{\kappa_1+\kappa_2}$, so that 
\eqref{e:p-ucvx2} becomes
\begin{eqnarray}
	\tfrac{c}{2}\tfrac{\kappa_1\kappa_2}{\kappa_1+\kappa_2}d(x,\Tbar x)^p
&\leq& \kappa_1 d(T_1x,x)^p+
\kappa_2 d(T_1x,\Tbar x)^p
\quad\forall x\in G.
\label{e:p-ucvx3}
\end{eqnarray}
It follows that  \eqref{eq:afne composition Fix TS}  holds for any  
$\kappabar\in (0,\tfrac{c\kappa_1\kappa_2}{2(\kappa_1+\kappa_2)}]$.  We conclude that 
the composition $\Tbar $ is quasi $\alpha$-firmly nonexpansive with constant 
\[
\alphabar=\frac{\kappa_1+\kappa_2}{\frac{c}{2}\kappa_1\kappa_2+\kappa_1+\kappa_2}. 
\]
A short calculation shows that this is the same as \eqref{eq:alphabar}, which completes the proof.
\end{proof}
\begin{remark}\label{r:compositions space}
	The fact that quasi $\alpha$-firm nonexpansiveness hinges on inequality 
	\eqref{eq:afne composition} or, more specifically \eqref{eq:afne composition Fix TS}, 
	is a property of the individual operators $T_2$ and $T_1$.  Whether or not the subsequent 
	inequality 
	\eqref{eq:afne composition Fix TS} holds is a property of the {\em space} and is entirely 
	independent of the operators.    Also note that the constant $\alphabar$ given 
	in \eqref{eq:alphabar} corresponds exactly to the constant found in 
	\cite[Proposition 4.44]{BauCom17} for mappings on Hilbert spaces.  
\end{remark}
\begin{corollary}[finite compositions of quasi $\alpha$-firmly nonexpansive 
operators are quasi $\alpha$-firmly nonexpansive]\label{t:m-compositionthm}
Let $(G, d)$ be a $p$-uniformly convex space.  
Let $T_1:D_1\to G$ where $D_1\subset G$ and for 
$j=2, 3, \dots,m$ let  
$T_j:D_{j}\to G$ for $D_j\equiv\{T_{j-1}x~|~x\in D_{j-1}\}$.  
If $T_j$ is quasi $\alpha$-firmly 
nonexpansive  
 with constant $\alpha_j$ on $D_j$ ($j=1,2,\dots,m$) and 
 $\Fix (T_m\circ T_{m-1}\circ\cdots\circ T_1)\subset D_1$ is nonempty,  
 then the composite operator 
 $T\equiv T_m\circ T_{m-1}\circ\cdots\circ T_1$ is quasi $\alpha$-firmly 
 nonexpansive  on $D_1$ with constant given recursively by 
 \begin{subequations}\label{e:alpha comp}
	 \begin{equation}\label{e:alphabarm}
		 \alphabar_{m}=\frac{\kappabar_{m-1}+\kappa_m}%
		 {\frac{c}{2}\kappabar_{m-1}\kappa_m+%
	 \kappabar_{m-1}+\kappa_m} \quad  (m\geq 3) 	  
	 \end{equation}
 where 
 \begin{eqnarray}\label{e:kappa(bar)j}
  \kappabar_j&=&\frac{1-\alphabar_j}{\alphabar_j} \quad (j\geq 2)\\
  \kappa_j&\equiv& \frac{1-\alpha_j}{\alpha_j}\quad  (j\geq 1)\\
  \alphabar_2&\equiv& \frac{\kappa_1+\kappa_2}{\tfrac{c}{2}\kappa_1\kappa_2+\kappa_1+\kappa_2}.
 \end{eqnarray}
 \end{subequations}
\end{corollary}
\begin{proof}
The result follows from Theorem \ref{t:compositionthm} and 
an elementary induction argument. 
\end{proof}

\begin{remark}\label{r:compositions Hilbert} It is well known that the composition of 
	two firmly nonexpansive mappings (for instance, projectors) in a Hilbert space 
	($\alpha=1/2$, $p=2$, and $c=2$)
	are $\alpha$-firmly nonexpansive with constant $\alphabar=\tfrac{2}{3}$.   
	Theorem \ref{t:compositionthm}
	yields this as a special case. 
\end{remark}

\subsection{Convex Combinations of $\alpha$-Firmly Nonexpansive Operators}
\label{s:convex combinations}
In this chapter we see that $\alpha$-formnes ist preserved under $p$-convex combinations of operators. To prove this we use the concept of $p$-uniformly convex functions.
\begin{mydef}
Let $(G,d)$ be a $p$-uniformly convex space. A function $f\colon G \rightarrow \Rbb$ is said to be $p$-uniformly convex with 
constant $m>0$ if
\begin{align*}
f(tx \oplus (1-t)y) \leq t f(x)+ (1-t) f(y) - \frac{1}{2} m t (1-t) d(x,y)^p 
\quad\forall x,y \in G, ~\forall t \in [0,1].
\end{align*}
\end{mydef}

\begin{remark}
It is obvious from the definition that the sum of two $p$-uniformly convex functions with constants $m_1$ and $m_2$ is $p$-uniformly convex with constant $m=m_1+m_2$. For any $y \in G$ the distance function $x \mapsto d(x,z)$ is a $p$-uniformly convex function with constant $m=c$ if $(G,d)$ is a $p$-uniformly convex space with constant $c>0$. 
\end{remark}

\begin{lemma} \label{minofstronglyconvexfunction}
Let $f\colon G \rightarrow \Rbb$ be $p$-uniformly convex with constant $m>0$ and $x \in \argmin f \neq \emptyset$. Then
\begin{align*}
f(y) \geq f(x) + \frac{m}{2} d(x,y)^p \quad \forall y \in G  \text{ and } x \in \argmin f. 
\end{align*}
\end{lemma}

\begin{proof}
Let $x \in \argmin f$ and $f$ be $p$-uniformly convex with constant $m$. Then
\begin{align*}
(1-t) f(y) &\geq f(tx \oplus (1-t)y)-t f(x) + \frac{m}{2} t (1-t) d(x,y)^p \\
& \geq (1-t) f(x) + \frac{m}{2} t (1-t) d(x,y)^p
\end{align*}
by the definition of $p$-uniformly convex functions and $x \in \argmin f$.
Now divide by $1-t$ and take the limit $t\rightarrow 1$ to obtain the claim.
\end{proof}

The $p$-convex combination of $n$ points $x_1, \ldots, x_n$ with weights 
$\omega_1,\omega_2, \ldots, \omega_n \in [0,1]$ such that $\sum_{i=1}^n \omega_i=1 $ is denoted $_p\!\!\oplus_i^n \omega_i x_i$ where  
\begin{align}
 _p\!\!\oplus_i^n \omega_i x_i\equiv  \argmin_y \sum_{i=1}^n \omega_i d(y,x_i)^p. \label{eq:pcvxcomb}
\end{align}  
For the convex combination of Operators $T_i$ is defined by
\begin{align}
\mathscr{T}x\coloneqq\argmin_y  \sum_{i=1}^n \omega_i d(y,T_i x)^p 
\label{eq:mathscrT}
\end{align}
and we denote $\mathscr{T} = _p\!\!\oplus_i^n \omega_i T_i x$.

Due to the next proposition $p$-convex combinations exist and are unique in complete p-uniformly convex spaces.

\begin{proposition}
Let $(G,d)$ be a complete $p$-uniformly convex space with constant $c>0$. Then the $\argmin$ in \eqref{eq:pcvxcomb} exists and is unique.
\end{proposition}
This is a special case of existence and uniqueness of p-barycenters in $p$-uniform convex spaces (see \cite[Lemma 3.5]{Kuwae2014}.
\begin{proof}
Let $y_k$ be a minimizing sequence of $y \mapsto \sum_{i=1}^n \omega_i d(y,x_i)^p$. Then \eqref{e:p-ucvx} yields
\begin{align*}
\sum_{i=1}^n \omega_i d(t y_k \oplus (1-t) y_l, x_i)^p \leq \sum_{i=1}^n \omega_i [
t d(y_k,x_i)^p + (1-t) d(y_l,x_i)^p- \frac{c}{2} t(1-t) d(y_k,y_l)^p
]
\end{align*}
rearranging with $t=\frac{1}{2}$ and yields 
\begin{align*}
\limsup_{k,l \rightarrow \infty} \frac{c}{8} d(y_k,y_l)^p &\leq \limsup_{k,l \rightarrow \infty} 
\frac{1}{2}\sum_{i=1}^n \omega_i d(y_k,x_i)^p
+\frac{1}{2}\sum_{i=1}^n \omega_i d(y_l,x_i)^p
-\sum_{i=1}^n \omega_i d(\frac{1}{2} y_k \oplus \frac{1}{2} y_l, x_i)^p \\
& \leq (\frac{1}{2}+\frac{1}{2}-1) \inf_y  \sum_{i=1}^n \omega_i d(y,x_i)^p =0.
\end{align*}
Hence $y_k$ is a Cauchy sequence and converges to a unique limit that is a minimizer of \eqref{eq:pcvxcomb}.
\end{proof}

\begin{mydef}[\cite{Kuwae2014}]
Let $(G,d)$ be a geodesic space. Let $\gamma$ and $\eta$ be two geodesics through $p$. Then $\gamma$ is said to be perpendicular to $\eta$ at point $p$ denoted by $\gamma \perp_p \eta$ if
\begin{align*}
d(x,p)\leq d(x,y) \quad \forall x\in \gamma, y \in \eta
\end{align*}
A space is said to be symmetric perpendicular if for all geodesics $\gamma$ and $\eta$ with common point $p$ we have
\begin{align*}
\gamma \perp_p \eta \Leftrightarrow \eta \perp_p \gamma.
\end{align*}
\end{mydef}

Examples for symmetric perpendicular spaces are $CAT(0)$ spaces and $CAT(\kappa)$ spaces for $\kappa>0$ with diameter strictly less than $\frac{\pi}{2 \sqrt{\kappa}}$ \cite[Theorem 2.11]{Kuwae2014}.

\begin{thm}
Let $(G,d)$ be a complete, p-uniformly convex space and 
for $i=1,2,\dots,n$ let the mappings $T_i:G\rightarrow G$ be  pointwise $\alpha$-firmly nonexpansive on 
$\Fix T_i$ with constant $\alpha_i$. Then for $\mathscr{T}$  defined by \eqref{eq:mathscrT}, 
$\cap_{i \in \{1,\ldots ,n\}}  \Fix T_i \subset \Fix \mathscr{T}$.  Suppose in addition  that 
$\cap_{i \in \{1,\ldots ,n\}} \Fix T_i \neq \emptyset$ and $G$ is symmetric perpendicular, then
$\Fix \mathscr{T}= \cap_{i \in \{1,\ldots ,n\}}  \Fix T_i $.
\end{thm}

\begin{proof}
The inclusion $\Fix \mathscr{T}\supset \cap_{i =1}^n  \Fix T_i$ is clear. 
To see the converse inclusion when the intersection  
$\cap_{i \in \{1,\ldots ,n\}} \Fix T_i \neq \emptyset$ and $G$ is symmetric perpendicular, 
let $x \notin \cap_{i=1}^n  \Fix T_i$ and 
$y \in \cap_{i=1}^n \Fix T_i $. 
For at least one $j \in \{1,\ldots ,n\}$ we have $x \notin \Fix T_j$.  
We use a contradiction to prove $P_{[x,y]}(T_j x)\neq x$. Therefore assume that $P_{[x,y]}(T_j x)=x$. 
Then $[x,T_jx] \perp_x [x,y]$ and by symmetric perpendicularity $[x,y] \perp_x [x,T_jx] $. 
Hence $d(y,x) \leq d(y,T_j x)$ this contradicts $d^p(y,T_j x)\leq d^p(x,y)- \frac{1-\alpha_i}{\alpha_j} \frac{c}{2} d^p(T_jx,x) < d^p(y,x) $. Therefore $t=0$ is not a minimum of the convex function 
$t\mapsto g_j(t)\coloneqq d(\mathscr{T}x,ty \oplus (1-t)x)^p$ on the interval $[0,1]$ and the right side 
derivative $d^+ g_j(0) <0$ for all $j$ with $T_jx \neq x$.  For $i$ with $T_i x=x$ we have $g_i(t)=t^p$ and hence $d^+ g_i(0)=0$. 
So the function 
\begin{align*}
g(t)\coloneqq \sum_{i=1}^n \omega_i d(ty \oplus (1-t)x,T_i x)^p = \sum_{i=1}^n \omega_i g_i(t)
\end{align*}
has $d^+g(0) <0$, and hence $x$ can not be a minimum of
\begin{align*}
z\mapsto \sum_{i=1}^n \omega_i d(z,T_i x)^p. 
\end{align*}
This shows that  $\mathscr{T}x \neq x$ and completes the proof.
\end{proof}

\begin{thm}[averages of pointwise $\alpha$-firm mappings are pointwise $\alpha$-firm]\label{t:cvx comb pafne} 
Let $(G,d)$ be a $p$-uniformly convex space with constant $c>0$ that is symmetric perpendicular. 
Let $T_i$ be pointwise $\alpha$-firmly nonexpansive with constant $\alpha_i$ ($i=1,2,\dots, n$) 
at all points in $\cap_{i=1}^n \Fix T_i \neq \emptyset$ 
on $D$, 
 and $\omega_i\in [0,1]$ with $\sum_{i=1}^n \omega_i=1$.
Then $\mathscr{T}$ defined by \eqref{eq:mathscrT}  is pointwise $\alpha$-firmly nonexpansive 
at all $y\in \Fix \mathscr{T}$ on $D$ with
\begin{align*}
\alpha = \max_i \alpha_i
\end{align*}
\end{thm}

\begin{proof}
Let $x \in D$.  By convexity of $d(\cdot,y)^p$ and Jensen's inequality \cite[Theorem 4.1]{Kuwae2014} for $p$-uniformly convex spaces with 
the symmetric perpendicular property we have
\begin{subequations}
\begin{align}
d(\mathscr{T}x,\mathscr{T}y)^p&=d(~ _p\!\!\oplus_i^n \omega_i T_i x,y)^p \\
&\leq ~ _p\!\!\oplus_i^n \omega_i d(T_i x,y)^p \\
&= \argmin_{t\in \Rbb} \sum_{i=1}^n \omega_i \vert t- d(T_ix,y)^p \vert^p \\
& \leq \argmin_{t\in \Rbb} \sum_{i=1}^n \omega_i \vert t -(d(x,y)^p-\frac{1-\alpha_i}{\alpha_i} \frac{c}{2} d(x,T_ix)^p)\vert^p \label{eq:1estimate} \\ 
& \leq \argmin_{t\in \Rbb} \sum_{i=1}^n \omega_i \vert t -(d(x,y)^p-\frac{1-\alpha}{\alpha}\frac{c}{2} d(x,T_ix)^p)\vert^p \label{eq:2estimate}\\
&= d(x,y)^p -\argmin_{t\in \Rbb} \sum_{i=1}^n \omega_i \vert t -\frac{1-\alpha}{\alpha} \frac{c}{2} d(x,T_ix)^p\vert^p \\
&\leq d(y,x)^p - \frac{1-\alpha}{\alpha} \frac{c}{2} d(x,~ _p\!\!\oplus_i^n \omega_i T_ix)^p \\
&=d(y,x)^p - \frac{1-\alpha}{\alpha} \frac{c}{2} d(x,\mathscr{T} x)^p
\end{align}
\end{subequations}
For the estimation in \eqref{eq:1estimate} and \eqref{eq:2estimate} we used the property that $\argmin_{t \in \Rbb} \sum_{i=1}^n \omega_i \vert t- \lambda_i\vert^p $ is increasing in every constant $\lambda_i$. This can be easily concluded since $\sum_{i=1}^n \omega_i \vert t- \lambda_i\vert^p $ is a convex function and 
\begin{align*}
\partial_t \sum_{i=1}^n \omega_i \vert t- \lambda_i\vert^p = \sum_{i=1}^n \omega_i p \vert t- \lambda_i\vert^{p-1} sgn(t-\lambda)
\end{align*} 
 is decreasing in every $\lambda_j$ for fixed $t$.
\end{proof}

\subsection{Constructing $\alpha$-firmly nonexpansive operators}

In a complete $p$-uniformly convex space the $p$-proximal mapping of a proper function lower semicontinuous $f$ is defined by
\begin{align}\label{e:prox^p}
\prox^p_{f, \lambda} (x) \coloneqq \argmin_{y \in G} f(y)+ \frac{1}{p \lambda^{p-1}} d(x,y)^p. 
\end{align}
The $\argmin$ in \eqref{e:prox^p} exists and is unique if $f$ is proper, lsc and convex 
\cite[Proposition 2.7]{Izuchukwu}.
This is a very natural definition of the proximal mapping, as the corresponding  Moreau-Yosida 
envelope given by
\begin{align*}
	e^p_{f,\lambda}(x) \coloneqq \inf_{y \in G} f(y) + \frac{1}{p \lambda^{p-1}} d(x,y)^p
\end{align*}
satisfies the semigroup property $e^p_{(e^p_{f,\lambda}),{\mu}}=e^p_{f, \lambda+\mu}$ 
(see \cite{Jost95} \cite{Kuwae2015}). 

\begin{proposition}[{\cite[Lemma 2.8]{Izuchukwu}}]\label{t:fne-type}
Let $(G,d)$ be a p-uniformly convex space with parameter $c> 0$, $\lambda>0$ and $f\colon G \rightarrow (-\infty, +\infty]$ be a proper, convex and lower semicontinuous function.  Then for all $x,y \in G$ we have
\begin{align*}
d(\prox^p_{f, \lambda}(x),\prox^p_{f, \lambda}(y))^p \leq \frac{1}{c} [d(v,y)^p+d(x,w)^p-d(x,v)^p-d(y,w)^p]=
\Delta_{\prox^p_{f, \lambda}}^{(p,\frac{4}{c})} (x,y)
\end{align*}
for $v=\prox^p_{f, \lambda}(x)$ and $w=\prox^p_{f, \lambda}(y)$.
\end{proposition}
\begin{proof}
This follows directly from \cite[Lemma 2.8]{Izuchukwu} with $\mu = \frac{p \lambda}{2}$.
\end{proof}

\begin{corollary}[proximal mappings are almost $\alpha$-firm]\label{t:prox aafne}
Let $(G,d)$ be a p-uniformly convex space with parameter $c\in (1,2]$, $\lambda>0$ 
and let $f\colon G \rightarrow (-\infty, +\infty]$ 
be a proper, convex, and lsc. Then $\prox^p_{f, \lambda}$ is almost $\alpha$-firmly nonexpansive with constant 
$\alpha_c = \frac{c-1}{c}$ and violation $\epsilon_c = \frac{2-c}{c-1}$.
\end{corollary}
\begin{proof}
Let $x \in \Fix \prox^p_{f, \lambda}$, $y \in G$ and $w = \prox^p_{f, \lambda}(y)$. 
Then by Proposition \ref{t:fne-type} and elementary calculations
\begin{align*}
d(x,w)^p \leq \frac{1}{c-1} (d(x,y)^p-d(y,w)^p)=(1+\epsilon_c) d(x,y)^p - \frac{1-\alpha_c}{\alpha_c } d(y,w)^p.
\end{align*}
\end{proof}

\begin{remark}
In the special case $c=2$ and hence $p=2$ the violation is $\epsilon_2=0$ and $\prox^2_{f, \lambda}$ is quasi $\alpha$-firm with constant $\alpha=\frac{1}{2}$. 
\end{remark}

\begin{proposition}[projectors are pointwise firmly nonexpansive]\label{t:projectors}
Let $(G,d)$ be a complete, symmetric perpendicular $p$-uniformly convex space, $C\subset G$ 
a convex subset. The metric projection onto the set $C$, denoted $P_C$, is pointwise 
$\alpha$-firmly nonexpansive at any $y\in C$ with constant $\alpha=\frac{1}{2}$.
\end{proposition}
\begin{proof}
First note that $[x,P_Cx] \perp_{P_Cx} [y,P_Cx]$ since $P_C$ is the metric projector. 
Then  $[y,P_Cx] \perp_{P_Cx} [x,P_Cx]$ by symmetric perpendicularity of the space. 
Hence $t=0$ is a minimum of the function $t \mapsto d(tx \oplus (1-t) P_Cx,y)^p$ on the interval $[0,1]$ and
\begin{align*}
d(tx \oplus (1-t) P_Cx,y)^p \leq t d(x,y)^p + (1-t) d(P_Cx,y)^p - \frac{c}{2} t (1-t) d(x,P_Cx)^p,
\end{align*}
with equality at $t=0$. Now $t=0$ has to be a minimum of the right hand side and
\begin{align*}
0&\leq  ~\dv{}{t} \Big\vert_{t=0} t d(x,y)^p + (1-t) d(P_Cx,y)^p - 
\frac{c}{2} t (1-t) d(x,P_Cx)^p\\
&=d(x,y)^p-d(P_Cx,y)^p-\frac{c}{2}d(x,P_Cx)^p,
\end{align*}
which yields the claim.
\end{proof}

\begin{proposition}[Krasnoselsky-Mann relaxations]
Let $(G,d)$ be a $p$-uniformly convex space and $T\colon G \rightarrow G$ 
be pointwise nonexpansive at all $y \in \Fix T$. Then 
$T_\lambda \coloneqq \lambda T \oplus (1-\lambda) Id$ is pointwise $\alpha$-firmly 
nonexpansive at all $y \in \Fix T$ with constant 
$\alpha= \frac{\lambda^{p-1}}{1-\lambda+\lambda^{p-1}}$.
\end{proposition}
\begin{proof}
Clearly $\Fix T = \Fix T_\lambda$ and $d(x,T_\lambda x)^p=\lambda^p d(x,Tx)^p$. 
Let $y \in Fix T_\lambda$ then 
\begin{align*}
d(y,T_\lambda x)^p &=d(y, \lambda Tx \oplus (1-\lambda)x)^p \\
&\leq \lambda d(y,Tx)^p+(1-\lambda) d(y,x)^p- \frac{c}{2} \lambda (1-\lambda) d(x,Tx)^p \\
& \leq d(x,y)^p - \frac{c}{2} \frac{1-\lambda}{\lambda^{p-1}} d(x,T_\lambda x)^p.
\end{align*}
Solving $\frac{1-\lambda}{\lambda^{p-1}}=\frac{1-\alpha}{\alpha}$ for $\alpha$ yields the claim.
\end{proof}

\section{Convergence of Iterated $\alpha$-Firmly Nonexpansive Mappings}
\label{s:convergence}
The  {\em asymptotic center} \cite{Ede72} of a bounded sequence 
$(x_k)_{k\in\Nbb}$ in a metric space $(G, d)$ is the set 
\begin{equation}\label{e:asymp center}
 A((x_k)_{k\in\Nbb})\equiv \set{x\in G}{\limsup_{k\to\infty}d(x,x_k) 
 = r((x_k)_{k\in\Nbb})}
\end{equation}
where 
\begin{equation}\label{e:asymp radius}
 r((x_k)_{k\in\Nbb})\equiv \inf\set{\limsup_{k\to\infty}d(y, x_k)}{y\in G}.
\end{equation}
Following \cite{Lim76}, a sequence $(x_k)_{k\in\Nbb}$ is said to 
$\Delta$-converge to $\xbar\in G$ whenever $\xbar$ 
is the unique asymptotic center of every subsequence of $(x_k)_{k\in\Nbb}$. 
In this case $\xbar$ is said to be the $\Delta$-limit of the sequence and we 
write $x_k\overset{\Delta}{\to}\xbar$.

\subsection{Convergence - No Rate}
The next theorem is a slight, but important generalization of analogous results 
that can be found elsewhere in the literature.  
There are two main differences:  namely,  that only 
{\em quasi} $\alpha$-firm nonexpansiveness is required, and secondly, 
nonexpansiveness is only required at the asymptotic centers of all subsequences.  
Our proof is nearly identical to the proof of   
\cite[Theorem 4.1]{RuiLopNic15}, but the stronger assumptions of the theorem 
of that work obscures the relationship between pointwise nonexpansiveness at
asymptotic centers and $\Delta$-convergence.  In both \cite{AriLeuLop14} and 
\cite{RuiLopNic15}, $\alpha$-firm nonexpansiveness implies 
nonexpansiveness, which is not the case here.  Moreover, in general it would 
be far too restrictive to require $\alpha$-firm nonexpansiveness everywhere when 
the property is really only required at its fixed points where there is still hope that 
the property enjoys a reasonable calculus.  
\begin{thm}\label{t:ne at asymp centers}
 Let $(G, d)$ be a $p$-uniformly convex space, let $D\subseteq G$ be convex, 
 and let $T: G\to G$ with $T(D)\subseteq D$ be pointwise $\alpha$-firmly 
 nonexpansive at all $y\in\Fix T\cap D$ on $D$.  Define the sequence 
 $(x_k)_{k\in \Nbb}$ by $x_{k+1}=Tx_k$ with $x_0\in D$.  If this sequence  
 is pointwise nonexpansive at the 
 asymptotic centers of all subsquences  on $D$, then the asymptotic centers
 of all subsequences coincide at a single $\xbar\in \Fix T\cap D$ and 
 $x_k\overset{\Delta}{\to}\xbar$.  
In particular, if $T$ is nonexpansive on $D$, then  
every fixed point sequence initialized in $D$ $\Delta$-converges to a point 
in $\Fix T$.
If, in addition, $T(D)$ is a boundedly compact subset of $G$,   then 
$x_k \to\xbar\in\Fix T$.   
\end{thm}
\begin{proof}
Let $\Ncal$ denote any infinite subset of $\Nbb$ and 
 consider the corresponding subsequence $(x_k)_{k\in\Ncal}$.  
 This subsequence is bounded since $T$ is a self-mapping on $D$ and 
 pointwise $\alpha$-firmly nonexpansive -- and hence by 
 Proposition \ref{t:properties pafne}\eqref{t:properties pafne iv}
 nonexpansive -- at all $y\in\Fix T\cap D$ on $D$.  Since $D$ is convex, this subsequence
therefore possesses a unique asymptotic center \cite{Leu10}, 
 which we denote by $\xbar_\Ncal$.  
Since $T$ is pointwise nonexpansive at $\xbar_\Ncal$ on $D$, we have
\begin{eqnarray*}
 \forall k\in\Ncal\qquad  d(T\xbar_\Ncal, x_{k})&\leq& d(T\xbar_\Ncal, Tx_k)+ d(Tx_k, x_k)\\
 &\leq & d(\xbar_\Ncal, x_k)+ d(Tx_k, x_k).
\end{eqnarray*}
Again, since $T$ is pointwise $\alpha$-firmly nonexpansive at all $y\in \Fix T\cap D$, 
by Theorem \ref{t:RLN THm 3.1} we have
$d(Tx_k, x_k)\to 0$ as $k\to\infty$.  
Therefore by \cite[Lemma 2.11]{AriLeuLop14} 
(see also \cite{Leu10}),  this implies that $T\xbar_\Ncal = \xbar_\Ncal$, that is, 
$\xbar_\Ncal\in\Fix T$.

Denote the unique asymptotic center of the entire sequence $(x_k)_{k\in\Nbb}$ by $\xbar$.  
Then
\begin{eqnarray*}
	\limsup_{k\underset{\Ncal}{\to}\infty}d(x_k, \xbar_\Ncal)&\leq& 
	\limsup_{k\underset{\Ncal}{\to}\infty}d(x_k, \xbar)\\
	&\leq& \limsup_{k\underset{\Nbb}{\to}\infty}d(x_k, \xbar)\\
	&\leq& \limsup_{k\underset{\Nbb}{\to}\infty}d(x_k, \xbar_\Ncal)\\
	&=&  \lim_{k\underset{\Nbb}{\to}\infty}d(x_k, \xbar_\Ncal)
	=  \lim_{k\underset{\Ncal}{\to}\infty}d(x_k, \xbar_\Ncal), 
\end{eqnarray*}
where the first equality follows from the fact that the sequence of distances is 
monotone decreasing and bounded below.  Therefore $\xbar_\Ncal=\xbar$.  
Since $\Ncal$ was an arbitrary infinite subset
of $\Nbb$, this establishes $\Delta$-convergence of $(x_k)$. 

To see strong convergence when $T(D)$ is boundedly compact, since $(x_k)_{k\in\Nbb}$
is a bounded sequence in $T(D)$, it has a convergent subsequence with 
limit $\xbar$.  Whenever $(d(x_k, \xbar))_{k\in\Nbb}$ converges, we can conclude 
that $x_k\to \xbar$. 
\end{proof}

\subsection{Quantitative Convergence - Error Bounds}
\label{s:convergence rate}
Our analysis of the convergence of fixed point iterations follows the same pattern developed in 
\cite{LukTamTha18, LukTebTha18, HerLukStu19a}.  In addition to pointwise 
$\alpha$-firm nonexpansiveness developed
above, we use the notion of {\em gauge monotonicity} of sequences and 
{\em metric subregularity}.
What we are calling gauge monotone sequences were first introduced 
 in \cite{LukTebTha18} where they are called
 $\mu$-monotone.  
 Recall that $\rho:[0,\infty) \to [0,\infty)$ is a \textit{gauge function} if 
$\rho$ is continuous, strictly increasing 
with $\rho(0)=0$, and $\lim_{t\to \infty}\rho(t)=\infty$. 

 \begin{mydef}[gauge monotonicity \cite{LukTebTha18}]\label{d:mu_Mon}
Let $(G,d)$ be a metric space, let $(x_k)_{k\in\Nbb}$ be a sequence on $G$, 
let $D\subset G$ be nonempty and let the continuous mapping
$\mymap{\mu}{\Rbb_+}{\Rbb_+}$ satisfy $\mu(0)=0$ and
\begin{align*}
&\mu(t_1)<\mu(t_2)\leq t_2\; \mbox{ whenever }\; 0\le t_1<t_2.
\end{align*}
\begin{enumerate}[(i)]
 \item $(x_k)_{k\in \Nbb}$ is said to be
 \emph{gauge monotone with respect to $D$ with rate $\mu$}  whenever
 \begin{equation}\label{e:mu-uniform mon}
 d(x_{k+1}, D)\leq \mu\paren{d(x_k, D)}\quad  \forall k\in \Nbb .
 \end{equation}
 \item $(x_k)_{k\in \Nbb}$ is said to be
 \emph{linearly monotone with respect to $D$} with rate $c$ if \eqref{e:mu-uniform mon} is
 satisfied for $\mu(t)=c\cdot t$ for all $t\in \Rbb_+$ and some constant $c\in [0,1]$.
 \end{enumerate}
A  sequence $(x_k)_{k\in \Nbb}$ is said to converge 
{\em gauge monotonically} 
to some element  $x^*\in G$ with rate 
$s_k(t)\equiv \sum_{j=k}^\infty \mu^{(j)}(t)$ whenever 
it is gauge monotone with gauge $\mu$ satisfying 
$\sum_{j=1}^\infty\mu^{(j)}(t)<\infty~\forall t\geq 0$, 
and 
there exists a constant 
$a>0$  such that $d(x_k,x^*)\leq  a s_k(t)$ for all $k\in \mathbb{N}$. 
\end{mydef}
All Fej\'er monotone sequences \cite{BauCom17} are linearly monotone (with constant $c=1$)
but the converse does not hold (see Proposition 1  and Example 1 of \cite{LukTebTha18}).
Gauge-monotonic convergence for a linear gauge in the definition above is just 
$R$-linear convegence.

The definition of metric subregularity below is modeled mainly after \cite[Definition 2.1b)]{Ioffe11} and  
\cite[Definition 1 b)]{Ioffe13}.  
\begin{mydef}[metric regularity on a set]\label{d:(str)metric (sub)reg}
$~$ Let $(G_1, d_1)$ and $(G_2,d_2)$ be metric spaces and let $\mymap{\Tcal}{G_1}{G_2}$, 
$U_1\subset G_1$, $U_2\subset G_2$.
For $\Lambda\subset G_1$, the mapping $\Tcal$ is called \emph{metrically regular on 
$U_1\times U_2$ relative to $\Lambda$ with gauge $\rho$} whenever
\begin{equation}\label{e:metricregularity}
d_1\paren{x, \Tcal^{-1}(y)\cap\Lambda}\leq \rho( d_2\paren{y, \Tcal(x)})
\end{equation}
holds for all $x\in U_1\cap\Lambda$ and $y\in U_2$ with $0<\rho(d_2\paren{y,\Tcal(x)})$ 
where $\Tcal^{-1}(y)\equiv \set{z}{\Tcal(z)=y}$.
When the set $U_2$ consists of a single point, $U_2=\{\bar y\}$, then $\Tcal$ is said 
to be 
\emph{metrically subregular for $\bar y$ on $U_1$ relative to $\Lambda$ with gauge $\rho$}.
\end{mydef}

The usual definition of metric subregularity is in the case where the 
gauge is just a linear function: $\rho(t)=\kappa t$.  The ``relative to'' part of the definition 
is also not common in the literature, but allows one to isolate the regularity 
to subsets (mostly manifolds) where the iterates of algorithms are naturally confined.  
See \cite[Example 1.8]{ACL16} for a concrete example.  
 In \cite[Example 3.9]{KohLopNic19} this is placed in a context of 
the {\em modulus of regularity} of a mapping with respect to its zeros. 
For our purposes, the easiest way to understand metric subregularity is as 
one-sided Lipschitz continuity of the (set-valued) inverse mapping $\Tcal^{-1}$.  
We will refer to the case when the gauge is linear to {\em linear metric subregularity}.

We construct $\rho$ implicitly from another 
nonnegative function $\mymap{\theta}{[0,\infty)}{[0,\infty)}$ satisfying
\begin{eqnarray}\label{eq:theta}
(i)~ \theta(0)=0; \quad (ii)~ 0<\theta(t)<t ~\forall t>0; 
\quad (iii)~\sum_{j=1}^\infty\theta^{(j)}(t)<\infty~\forall t\geq 0.
\end{eqnarray}
The gauge we will use satisfies 
\begin{equation}\label{eq:gauge}
 \rho\paren{\paren{\frac{t^p-\paren{\theta(t)}^p}{\tau}}^{1/p}}=
 t\quad\iff\quad
 \theta(t) = \paren{t^p - \tau\paren{\rho^{-1}(t)}^p}^{1/p}
\end{equation}
for $\tau>0$ fixed and $\theta$ satisfying \eqref{eq:theta}.  

In the case of linear metric subregularity on a $2$-uniformly convex space (think Hilbert space)
we have 
\[
\rho(t)=\kappa t\quad\iff\quad  
\theta(t)=\paren{1-\frac{\tau}{\kappa^2}}^{1/2}t\quad (\kappa\geq \sqrt{\tau}).
\]  
The condition $\kappa\geq \sqrt{\tau}$ is spurious since, if 
\eqref{e:metricregularity} is satisfied for some $\kappa'>0$, then it is satisfied
for all $\kappa\geq \kappa'$. 

From the transport discrepancy $\psi^{(p,c)}_{T}$ defined in \eqref{eq:psi}  
and a subset $S\subset G$ we construct
the following surrogate mapping 
$\mymap{\Tcal_S}{G}{\mathbb{R}_+}\cup\{+\infty\}$ 
 by 
 \begin{equation}\label{eq:Tcal}
	 \Tcal_S(x)\equiv \left(\tfrac{2}{c}\inf_{y\in S}\psi^{(p,c)}_{T}(x,y) \right)^{1/p}.
 \end{equation}
 If $S=\emptyset$ then, by definition, $\Tcal_S(x)\equiv+\infty$ for all $x$.  
 When $S\subseteq\Fix T$, then by Proposition \ref{t:properties pafne}\eqref{t:properties pafne ii} 
 \begin{equation}\label{eq:Tcal_Fix_T}
	 \Tcal_S(x)= \sqrt[p]{\tfrac{2}{c}}d(Tx,x)>0\quad (S\neq\emptyset).  
 \end{equation}
 Hence, this function is proper (finite at least at one point, and does not take the value $-\infty$) 
 when  $S\subseteq\Fix T$ is nonempty.   
This can be interpreted as the pointwise transport discrepancy relative to the fixed points and 
will be used to characterize the regularity of the mapping $T$ at fixed points.  

\begin{thm}[quantitative convergence]\label{t:msr convergence}
 Let $(G, d)$ be a 
 $p$-uniformly convex space, let $D\subset G$, 
 let $T:G\to G$ with $T(D)\subseteq D$ boundedly compact,  
 and let $S\equiv \Fix T\cap D$ be nonempty. 
Assume 
 \begin{enumerate}[(i)]
 \item\label{t:msr convergence a} $T$ is pointwise $\alpha$-firmly nonexpansive at all 
 points $y\in S$ 
 with the same constant $\alphabar$ on $D$;
 \item\label{t:msr convergence b} $\Tcal_S$ defined by \eqref{eq:Tcal} is metrically subregular for 
 $0$ relative to $D$ on $D$ with gauge $\rho$ given by \eqref{eq:gauge} for 
 $\tau=c(1-\alphabar)/(2\alphabar)$, that is, 
 \begin{equation}
 \label{eq:estimate1}
 d(x,\Fix T\cap D)\leq \rho(d(Tx,x)),\hspace{0.2cm}\forall x\in D.
 \end{equation}
 \end{enumerate}
Then for any $x_0\in D$, the sequence $(x_k)_{k\in\Nbb}$ defined by 
$x_{k+1}\equiv T x_k$ satisfies 
\begin{equation}\label{eq:gauge convergence}
d\paren{x_{k+1},\Fix T\cap D}
\leq \theta\paren{d\paren{x_k,\Fix T\cap D}} 
\quad \forall k \in \mathbb{N},
\end{equation}%
where $\theta$ given implicitly by \eqref{eq:gauge} satisfies \eqref{eq:theta}. 
Moreover, the sequence $(x_k)_{k\in\Nbb}$
converges gauge monotonically  to 
some $x^{*}\in\Fix T\cap D$ with rate 
$O(s_k(t_0))$ where 
$s_k(t)\equiv 
\sum_{j=k}^\infty \theta^{(j)}(t)$ and $t_0\equiv d(x_0,\Fix T\cap D)$.
 \end{thm}

 Before proving the result, we establish convergence of gauge monotone 
 sequences.   
 
\begin{lemma}[gauge monotonicity and quasi $\alpha$-firmness imply convergence
to fixed points]
\label{t:rm and qafne to convergence}
Let $(G, d)$ be a complete, $p$-uniformly convex metric space with constant $c$.  
Let $T:G\to G$ with  $T(D)\subseteq D\subseteq G$ and $T(D)$ boundedly compact.
Suppose that $\Fix T\cap D$ is nonempty and that $T$ is pointwise 
$\alpha$-firmly nonexpansive at all $y\in \Fix T\cap D$ with the same constant 
$\alphabar$ on $D$. If the sequence $(x_k)_{k\in\mathbb{N}}$ defined by 
$x_{k+1}=Tx_k$ and initialized in $D$ 
is gauge monotone relative to $\Fix T\cap D$ 
with rate $\theta$ satisfying \eqref{eq:theta},
then $(x_k)_{k\in\mathbb{N}}$ converges gauge monotonically  to 
some $x^*\in\Fix T\cap D$ with rate $O(s_k(t_0))$ where 
$s_k(t)\equiv 
\sum_{j=k}^\infty \theta^{(j)}(t)$ and $t_0\equiv d(x_0,\Fix T\cap D)$.
\end{lemma}
\begin{proof}
By \eqref{e:psi-Fix T}, the assumption that $T$ is pointwise $\alpha$-firmly nonexpansive at all
 $y\in\Fix T\cap D$ with constant $\alphabar$ on $D$ yields 
 $$d(Tx,y)^p\leq  d(x,y)^p-\tfrac{c(1-\alphabar)}{2\alphabar}d(x,Tx)^p,
 \hspace{0.2cm}\forall x\in D.$$
 Let $x_0\in D$ and define the sequence $x_{k+1}:=Tx_k$ for all $k\in\mathbb{N}$.
 Since $T(D)$ is boundedly compact and $T$ is pointwise $\alpha$-firmly nonexpansive 
 at all points in $\Fix T\cap D$ on $D$, by Proposition \ref{t:properties pafne}\eqref{t:properties pafne iv} 
 and Lemma \ref{l:fixpointcvx},
 $P_{\Fix T\cap D}x_k$ is nonempty (though possibly set-valued) for all $k$;  
 denote any selection by $\bar{x}_k\in P_{\Fix T\cap D}x_k$ for each $k\in\mathbb{N}$. 
 Then we have
 $$d(x_{k+1},\bar{x}_k)^p\leq  d(x_k,\bar{x}_k)^p-\tfrac{c(1-\alphabar)}{2\alphabar}
 d(x_k,x_{k+1})^p,\hspace{0.2cm}\forall k\in \mathbb{N},$$
 which implies  that
 $$d(x_k,x_{k+1})\leq  \paren{\tfrac{c(1-\alphabar)}{2\alphabar}}^{-1/p}d(x_k,\bar{x}_k),
 \hspace{0.2cm}\forall k\in \mathbb{N}.$$
 On the other hand 
 $d(x_k,\bar{x}_k)=
 d(x_k,\Fix T\cap D)
 \leq  \theta\paren{d(x_{k-1},\Fix T\cap D)}$ 
 since $(x_k)_{k\in\mathbb{N}}$ is gauge monotone relative to $\Fix T\cap D$ 
 with rate $\theta$. Therefore an iterative application of gauge monotonicity yields
 $$d(x_k,x_{k+1})
 \leq   \paren{\tfrac{c(1-\alphabar)}{2\alphabar}}^{-1/p}\theta^{(k)}\paren{d(x_0, \Fix T\cap D)},
 \hspace{0.2cm}\forall k\in \mathbb{N}.$$
 Let $t_0=d(x_0, \Fix T\cap D)$.
 For any given natural numbers $k,l$ with $k<l$ an iterative application of the triangle 
 inequality yields the upper estimate 
\begin{eqnarray*}
 d(x_k,x_l)&\leq& d(x_k,x_{k+1})+d(x_{k+1},x_{k+2})+...+d(x_{l-1},x_l)\\
 &\leq& \paren{\tfrac{c(1-\alphabar)}{2\alphabar}}^{-1/p}
 \paren{ \theta^{(k)}(t_0)+\theta^{(k+1)}(t_0)+\dots+\theta^{(l-1)}(t_0)}\\
 &<&  \paren{\tfrac{c(1-\alphabar)}{2\alphabar}}^{-1/p} s_{k}(t_0),
\end{eqnarray*}
where $s_k(t_0)\equiv \sum_{j=k}^{\infty}\theta^{(j)}(t_0)<\infty$ 
for $\theta$ satisfying \eqref{eq:theta}.  
Since $(\theta^{(k)}(t_0))_{k\in\Nbb}$ is 
a summable sequence of nonnegative numbers, the sequence of 
partial sums $s_{k}(t_0)\to 0$ monotonically as $k\to\infty$
and hence $(x_k)_{k\in\mathbb{N}}$ 
is a Cauchy sequence. Because $(G,d)$ is a complete metric space we conclude that $x_k\to x^*$ for some 
$x^*\in G$.  Letting $l\to+\infty$ yields
$$\lim_{l\to+\infty}d(x_k,x_l)=d(x_k,x^*)\leq  a s_{k}(t_0),
\hspace{0.2cm}a\equiv\paren{\frac{c(1-\alphabar)}{2\alphabar}}^{-1/p}.$$
Therefore $(x_k)_{k\in\mathbb{N}}$ converges gauge monotonically to $x^*$ with rate $O(s_k(t_0))$.

It remains to show that $x^*\in\Fix T\cap D$. Note that for each $k\in \mathbb{N}$ we have 
$$d(x_k,\bar{x}_k)=d(x_k,\Fix T\cap D)\leq  \theta^{(k)}(t_0),$$
which yields $\lim_kd(x_k,\bar{x}_k)=0$. 
But by the triangle inequality 
$$d(\bar{x}_k,x^*)\leq  d(x_k,\bar{x}_k)+d(x_k,x^*),$$
so $\lim_kd(\bar{x}_k,x^*)=0$. By construction 
$(\bar{x}_k)_{k\in\mathbb{N}}\subseteq \Fix T\cap D$ and by Lemma \ref{l:fixpointcvx}  
$\Fix T\cap D$ is closed, hence $x^*\in\Fix T\cap D$. 
\end{proof}
 
 \noindent{\em Proof of Theorem \ref{t:msr convergence}. }
Since $S=\Fix T\cap D$, by Proposition \ref{t:properties pafne}\eqref{t:properties pafne ii} we have  
$\psi^{(p,c)}_{T}(x,y)=\tfrac{c}{2}d(Tx,x)^p$ for all $y\in \Fix T$, so
in fact $\Tcal_S(x)=d(Tx,x)$.  Also by Proposition \ref{t:properties pafne}\eqref{t:properties pafne ii} 
$\Tcal_S$ takes the 
value $0$ only on $\Fix T$, that is, $\Tcal_S^{-1}(0)=\Fix{T}$.  
So by assumption 
\eqref{t:msr convergence b} and the definition of metric subregularity
(Definition \ref{d:(str)metric (sub)reg}) 
\begin{eqnarray}
d(x, \Fix T\cap D) &=& 
d(x, \Tcal_S^{-1}(0)\cap D)
  \nonumber\\
 &\le& 
\rho\paren{|\Tcal_S(x)|}
 = \rho(    d(Tx, x))\quad\forall x\in D.\nonumber
\end{eqnarray}
In other words, 
\begin{equation}\label{eq:rate step 1}
   \tfrac{1-\alphabar}{\alphabar}\tfrac{c}{2}
   \paren{\rho^{-1}\paren{%
   d(x,\Fix T\cap D)}}^{p}
\leq \tfrac{1-\alphabar}{\alphabar}\tfrac{c}{2}d(Tx, x)^p
\quad\forall x\in D.
\end{equation}
On the other hand, by assumption 
\eqref{t:msr convergence a} we
have 
\begin{eqnarray}
  \tfrac{1-\alphabar}{\alphabar}\tfrac{c}{2}d(Tx, x)^p\leq d(x,y)^p-d(Tx,y)^p
      \quad\forall y\in  \Fix T\cap D , 
      \forall x\in D.
      \label{eq:rate step 2}
\end{eqnarray}
Incorporating \eqref{eq:rate step 1} into 
\eqref{eq:rate step 2} and rearranging the inequality yields
\begin{eqnarray}
d(Tx,y)^p
   \!&\leq&\! 
      d(x,y)^p - 
         \tfrac{1-\alphabar}{\alphabar}\tfrac{c}{2}
   \paren{\rho^{-1}\paren{%
    d(x, \Fix T\cap D)}}^{p}
      \quad\forall y\in \Fix T\cap D , 
      \forall x\in D.
      \label{eq:rate step 3}
\end{eqnarray}
Since this holds at {\em any} $x\in D$, it certainly 
holds at the iterates $x_k$ with initial point $x_0\in D$
since $T$ is a 
self-mapping on $D$.   Therefore 
\begin{equation}\label{eq:gauge convergence 0}
d\paren{x_{k+1},\, y}
\leq \sqrt[p]{ d\paren{x_{k},\,y}^p - 
\frac{1-\alphabar}{\alphabar}\tfrac{c}{2}
\paren{\rho^{-1}\paren{%
 d\paren{x_{k},\, \Fix T\cap D}}}^p} 
\quad\forall y\in \Fix T\cap D, ~
\forall k \in \mathbb{N}.
\end{equation}%

Equation \eqref{eq:gauge convergence 0} simplifies.  
Indeed, by Lemma \ref{l:fixpointcvx}, $\Fix T\cap D $ is closed.  
Moreover, since $T(D)$ is assumed to be boundedly compact, 
for every $k\in\Nbb$  the distance  $d(x_k, \Fix T\cap D)$ 
is attained  at some $y_k\in \Fix T\cap D$ yielding
\begin{equation}
d(x_{k+1},y_{k+1} )^p \leq 
d(x_{k+1},y_{k} )^p \leq 
d(x_k, y_k)^p  - 
\tfrac{1-\alphabar}{\alphabar}\tfrac{c}{2}\paren{\rho^{-1}
\paren{ d(x_k, y_k)}}^p
\quad\forall k \in \mathbb{N}.
\label{eq:gauge convergence intermed}
\end{equation}
Taking the $p$-th root and recalling \eqref{eq:gauge}
yields \eqref{eq:gauge convergence}.

This establishes also that the sequence $(x_k)_{k\in\Nbb}$ is 
gauge monotone relative to $\Fix T\cap D$ with rate 
$\theta$ satisfying 
Eq.\eqref{eq:theta}.  
By Lemma \ref{t:rm and qafne to convergence} we conclude that 
the sequence $(x_k)_{k\in\Nbb}$ converges 
gauge monotonically to $x^*\in\Fix T\cap D$ with the rate
$O(s_k(d(x_0,\Fix T\cap D)))$ where $s_k(t)\equiv 
\sum_{j=k}^\infty \theta^{(j)}(t)$. 
\hfill $\Box$

In \cite[Theorem 2]{LukTebTha18} it is shown that if every fixed point sequence initialized on 
$D\subset G$ is linearly monotone with respect to $\Fix T\cap D$ with rate $c<1$ then the surrogate
mapping $\Psi$ is linearly metrically subregular for $0$ relative to $D$ on $D$.  
From this they establish that 
linear metric subregularity is in fact necessary for linear convergence of fixed point 
sequences generated by almost $\alpha$-firmly nonexpansive mappings \cite[Corollary 1]{LukTebTha18}.  
We show that 
this extends more generally to fixed point iterations in $p$-uniform metric spaces of 
quasi $\alpha$-firmly nonexpansive mappings where the iterates converge at a rate
characterized by $\theta$.  

\begin{thm}[necessity of metric subregularity for monotone sequences]
\label{t:msr necessary}
Let $(G, d)$ be a $p$-uniformly convex metric space with constant $c$.  
Let $T:D\to D$ with  $D\subseteq G$.
Suppose that $S\equiv \Fix T\cap D$ is nonempty.  Suppose  
all sequences $(x_k)_{k\in\mathbb{N}}$ defined by $x_{k+1}=Tx_k$ 
and initialized in $D$ are gauge monotone 
relative to $S$ 
with rate $\theta$ satisfying \eqref{eq:theta}.  Suppose, in addition, that  
$(\Id - \theta)^{-1}(\cdot)$ is continuous on $\Rbb_+$, strictly increasing, 
and $(\Id - \theta)^{-1}(0)=0$.  Then $\Tcal_{S}$ defined by 
\eqref{eq:Tcal} is metrically subregular for $0$ relative to $D$ on $D$ 
with gauge $\rho(\cdot)=(\Id-\theta)^{-1}(\cdot)$.
\end{thm}
\begin{proof}
If the fixed point sequence is gauge monotone
relative to $S$ 
with rate $\theta$ satisfying \eqref{eq:theta} then by the triangle inequality 
\begin{eqnarray}
d(x_{k+1},x_k)&\geq& d(x_k, S) - d(x_{k+1}, S) \nonumber\\
&\geq& d(x_k, S) - \theta\paren{d(x_k, S)} \quad  \forall k\in \Nbb.
\label{e:Robin}
\end{eqnarray}
On the other hand, as shown in the proof of Theorem \ref{t:msr convergence}
\begin{eqnarray}
\Tcal_{S}^{-1}(0) &=& \Fix T, \nonumber\\
d(0,\Tcal_{S}(x_k))& =& d(x_{k+1}, x_k) \label{e:Hood}
\end{eqnarray}
Combining \eqref{e:Robin} and \eqref{e:Hood} yields
\begin{equation}\label{e:dumb}
d(0,\Tcal_{S}(x_k))\geq d(x_k, \Tcal_{S}^{-1}(0)\cap D) - 
\theta\paren{d(x_k, \Tcal_{S}^{-1}(0)\cap D)} \quad  \forall k\in \Nbb
 \end{equation}
By assumption $(\Id - \theta)^{-1}(\cdot)$ is continuous on $\Rbb_+$, 
strictly increasing, 
and $(\Id - \theta)^{-1}(0)=0$, so 
\begin{equation}\label{e:dumber}
(\Id - \theta)^{-1}\paren{d(0,\Tcal_{S}(x_k))}\geq d(x_k, \Tcal_{S}^{-1}(0)\cap D) 
\quad  \forall k\in \Nbb.
 \end{equation}
Since this holds for {\em any} sequence $(x_k)_{k\in\Nbb}$ initialized in $D$, 
we conclude that $\Tcal_S$ is metrically subregular for $0$ on $D$ with 
gauge $\rho=(\Id - \theta)^{-1}$.  
\end{proof}

The next corollary is an immediate consequence of Lemma \ref{t:rm and qafne to convergence}
and Theorem \ref{t:msr necessary}.  
\begin{corollary}[necessity of metric subregularity for gauge monotone convergence]
\label{t:msr necessary convergence}
Let $(G, d)$ be a $p$-uniformly convex metric space with constant $c$.  
Let $T:D\to D$ with  $D\subseteq G$.
Suppose that $S\equiv \Fix T\cap D$ is nonempty and that $T$ is 
$\alpha$-firmly nonexpansive at all $y\in S$ on $D$.  Suppose that 
all sequences $(x_k)_{k\in\mathbb{N}}$ defined by $x_{k+1}=Tx_k$ 
and initialized in $D$ are gauge monotone 
relative to $S$ 
with rate $\theta$ satisfying \eqref{eq:theta}. 
Suppose, in addition, that  
$(\Id - \theta)^{-1}(\cdot)$ is continuous on $\Rbb_+$, strictly increasing, 
and $(\Id - \theta)^{-1}(0)=0$.  Then all sequences initialized on $D$ 
converge gauge monotonically to some $\xbar\in S$ with rate 
$O(s_k(t_0))$ where 
$s_k(t)\equiv 
\sum_{j=k}^\infty \theta^{(j)}(t)$ and $t_0\equiv d(x_0,\Fix T\cap D)$.
Moreover, $\Tcal_{S}$ defined by 
\eqref{eq:Tcal} is metrically subregular for $0$ relative to $D$ on $D$ 
with gauge $\rho(\cdot)=(\Id-\theta)^{-1}(\cdot)$. 
\end{corollary}

\section{Examples}\label{s:Examples}
Most of the concrete examples provided here are for $p$-uniformly convex spaces with $p=c=2$, i.e.
CAT(0) spaces, and these are mostly known.  We hint at a path beyond this setting and in the case
of cyclic projections obtain an extension of \cite[Proposition 4.1]{RuiLopNic15} to 
complete, symmetric perpendicular, $p$-uniformly convex spaces.

\subsection{Proximal Splitting}
Let $(H,d)$ be a Hadamard space,  $f_i:H\to H$ be proper lsc 
convex functions for $i=1,2,\dots N$.  Consider the problem
\begin{equation}
 \label{eq:sum of cvx}
 \inf_{x\in H}\sum_{i=1}^N f_i(x).
 \end{equation}
 In this setting, the $p$-proximal mapping of $f$ \eqref{e:prox^p} simplifies to 
 \begin{equation}\label{eq:prox}
	\prox^2_{f, \lambda}(x)\equiv\argmin_{y\in H}\klam{f(y)+\tfrac{1}{2\lambda}d(x,y)^2}.
 \end{equation}
 This has been studied in CAT(0) spaces in \cite{Jost97, Banert, AriLeuLop14} and in the 
 Hilbert ball in \cite{KopRei09}.  To reduce notational clutter, we drop the superscript $2$.  
 In these earlier works it was already known that 
resolvents of lsc convex functions are (everywhere) $\alpha$-firmly nonexpansive with 
$\alpha=1/2$.  The specialization of Corollary \ref{t:prox aafne} to the case $p=c=2$
confirms this.   Applying {\em backward-backward splitting} to this problem yields 
 Algorithm \ref{alg:bbs}.
\begin{algorithm}[h]    
\SetKwInOut{Input}{Parameters}\SetKwInOut{Output}{Initialization}
  \Input{Functions $f_1\ldots,f_N$ and $\lambda_i >0$ $(i=1,2,\dots,N)$.}
  \Output{Choose  $x_0\in H$.}
  \For{$k = 0,1, 2, \ldots $}{
    \begin{align*}
		x_{k+1}=Tx_k\equiv 
		\paren{\prox_{f_N, \lambda_N}\circ\cdots\circ \prox_{f_2, \lambda_2}\circ \prox_{f_1, \lambda_1}}(x_k)
    \end{align*}
}
\caption{
Proximal splitting}\label{alg:bbs}
\end{algorithm}
We are certainly not the first to study this algorithm.  Indeed, convergence has been established already in 
\cite[Theorem 4.1]{RuiLopNic15}.   This conclusion 
 also follows immediately from Theorem \ref{t:ne at asymp centers}  upon
 application of Corollary \ref{t:m-compositionthm} which 
 shows that the composition of quasi-$\alpha$-firmly nonexpansive $\prox$ mappings, 
 $ \prox_{f_i, \lambda_i}$, is quasi-$\alpha$-firmly 
 nonexpansive on $H$ 
 with constant $\alphabar_N$ given recursively by \eqref{e:alpha comp}.  
 If on a neighborhood of $\Fix T$, denoted by $D$, the  
 mapping $\Tcal_{\Fix T\cap D}$ defined by \eqref{eq:Tcal} -- which by Proposition 
 \ref{t:properties pafne}\eqref{t:properties pafne ii} simplifies to \eqref{eq:Tcal_Fix_T} -- 
 satisfies 
 \begin{equation}\label{eq:simple msr}
  d(x,\Fix T\cap D)\leq \rho( d(Tx, x))\quad \forall x\in D 
\end{equation}
where $\rho$ is a gauge 
given by \eqref{eq:gauge} for $\tau = \tfrac{1-\alphabar_N}{\alphabar_N}$, 
 then by Theorem \ref{t:msr convergence} the sequence $(x_k)$ converges gauge monotonically 
 to some $x^{*}\in\Fix T$ with rate 
$O(s_k(t_0))$ where 
$s_k(t)\equiv 
\sum_{j=k}^\infty \theta^{(j)}(t)$ and $t_0\equiv d(x_0,\Fix T)$
for  $\theta$ given implicitly by \eqref{eq:gauge}. 

By Corollary \ref{t:prox aafne}, on spaces with curvature bounded above, the $p$-proximal mapping 
is only {\em almost} $\alpha$-firmly nonexpansive, which then yields that the composition of 
$p$-proximal mappings is also only almost $\alpha$-firmly nonexpansive.  However, the violation 
$\epsilon_c = \frac{2-c}{c-1}$, where $c$ is the constant of curvature of the space.  This constant
can be made arbitrarily small by choosing a small enough domain.  In this way, the violation can also 
be made arbitrarily small.  As shown in \cite{LukTamTha18, LukMar20} in the context of Euclidean 
spaces, if $\Tcal_{\Fix T}$
is metrically subregular, then the violation of $\alpha$-firm nonexpansiveness can be overcome 
to yield quantifiable (e.g. linear) convergence on neighborhoods of $\Fix T$.  This 
would then yield for the first time convergence of proximal splitting algorithms on spaces with 
positive curvature.   This will be the subject of a future study.  

\subsection{Projected Gradients}
Here we specialize problem \eqref{eq:sum of cvx} to the case $N=2$ and 
$f_2=\iota_C$, the indicator function of some closed convex set $C\subset H$.  Recall, 
in a Hadamard space Moreau-Yosida envelope of $f$ is defined by 
\[
e_{f, \lambda}(x)\equiv \inf_{y\in H}\paren{f(y)+\tfrac{1}{2\lambda}d(x,y)^2}. 
\]
In a Hilbert space setting, the 
proximal mapping of a convex function $f$ and the resolvent of its subdifferential are
one and the same.  Moreover, $e_{f, \lambda}$ is continuously 		
differentiable with  $\nabla e_{f, \lambda} = \frac{1}{\lambda}\paren{\Id - \prox_{f, \lambda}}$.  
A step of length $\tau$ 
in the direction of steepest descent of the Moreau-Yosida envelope of $f$ takes the form 
\[
x-\tau \nabla e_{f, \lambda}(x) = \paren{(1-\tau)\Id + \tau\prox_{f, \lambda}}(x).
\]
Formally transposing this to a CAT(0) space yields the nonlinear analog to the 
direction of steepest descent for $e_{f, \lambda}$:
\begin{equation}\label{e:sd Me}
	(1-\tau)x\oplus \tau \prox_{f,\lambda}(x).
\end{equation}
This leads to Algorithm \ref{alg:pg},  the analog to projected gradients in  CAT(0) space, 
which is nothing more than a projected resolvent/ projected proximal iteration. 
\begin{algorithm}[h]    
\SetKwInOut{Input}{Parameters}\SetKwInOut{Output}{Initialization}
\Input{$\mymap{f}{H}{\Rbb}$, the closed set $C\subset H$, $\lambda>0$ and $\tau\in(0,1)$.}
  \Output{Choose  $x_0\in H$.}
  \For{$k = 0,1, 2, \ldots $}{
    \begin{align*}
		& x_{k+1}=T_{PG}(x_k)\equiv P_{C}\paren{(1-\tau)\Id\oplus \tau \prox_{f,\lambda}}(x_{k})
    \end{align*}
}
\caption{Metric Projected Gradients}\label{alg:pg}
\end{algorithm}
Theorem \ref{t:cvx comb pafne} establishes that the mapping 
$x\mapsto \paren{(1-\tau)\Id\oplus \tau \prox_{f,\lambda}}$
is $\alpha$-firmly 
nonexpansive with constant $\alphabar=1/2$.  Therefore, by 
Theorem \ref{t:compositionthm} the operator $T_{PG}$ is $\alpha$-firmly nonexpansive on $H$ 
with constant 
$\alphahat= \tfrac{2}{3}$.
Theorem \ref{t:ne at asymp centers} then guarantees that the sequence $(x_k)$ is 
 $\Delta$-convergent to some $x^*\in\Fix T_{GF}$, 
with strong convergence whenever $T_{GF}$ is boundedly compact.  
If in addition \eqref{eq:simple msr} is satisfied with $T$ replaced by $T_{PG}$ and 
with gauge $\rho$ given by \eqref{eq:gauge} for $\tau = 1/2$, 
 then, again,  by Theorem \ref{t:msr convergence}  the sequence $(x_k)$ converges gauge monotonically 
 to 
some $x^{*}\in\Fix T$ with rate 
$O(s_k(t_0))$ where 
$s_k(t)\equiv 
\sum_{j=k}^\infty \theta^{(j)}(t)$ and $t_0\equiv d(x_0,\Fix T)$
for  $\theta$ given implicitly by \eqref{eq:gauge}. 

\subsection{Cyclic Projections in $p$-uniformly Convex Spaces}
For compositions of projectors we are not confined to Hadamard spaces.  We consider 
Algorithm \ref{alg:bbs} when the functions $f_i\equiv \iota_{C_i}$, the indicator functions 
of closed convex sets $C_i\subset G$, where $(G, d)$ is a 
complete, symmetric perpendicular $p$-uniformly convex space with constant $c$.  
The $p$-proximal mapping of the indicator  function is the metric projector and so by 
Proposition \ref{t:projectors} these are pointwise $\alpha$-firmly nonexpansive at all 
points in $\cap_i C_i$ (assuming, of course, that this is nonempty).  
By Lemma \ref{t:afne of compositions} the cyclic projections mapping
\begin{equation}\label{e:cp mapping}
	T_{CP}\equiv P_{C_N}\cdot P_{C_2} P_{C_1}
\end{equation}
is pointwise $\alpha$-firmly nonexpansive at all 
points in $\cap_i C_i=\Fix T_{CP}$, when the intersection is nonempty, 
with constant $\alphabar_N = \frac{N-1}{N}$
on $G$.  The
only asymptotic centers of subsequences of cyclic projections are points
in this intersection, and here the projectors, and hence the cyclic projections mapping,  
are pointwise nonexpansive. 
So by Theorem \ref{t:ne at asymp centers} the cyclic projections sequence 
$\Delta$-converges to a point in  $\cap_i C_i$ whenever this is nonempty, and 
converges strongly whenever at least one of the sets $C_i$ is compact.  This 
generalizes \cite[Proposition 4.1]{RuiLopNic15} which is limited to CAT($\kappa$)
spaces (i.e. $p=2$, $c<2$ small enough).  

If in addition 
 \begin{equation}\label{eq:simple msr - sets}
	 d(x,\cap_i C_i)\leq \rho( d(T_{CP}x, x))\quad \forall x\in G 
\end{equation}
where $\rho$ is a gauge 
given by \eqref{eq:gauge} for $\tau = \tfrac{1}{N-1}$, 
 then by Theorem \ref{t:msr convergence} the sequence $(x_k)$ converges gauge monotonically 
 to some $x^{*}\in\Fix T$ with rate 
$O(s_k(t_0))$ where 
$s_k(t)\equiv 
\sum_{j=k}^\infty \theta^{(j)}(t)$ and $t_0\equiv d(x_0,\Fix T)$
for  $\theta$ given implicitly by \eqref{eq:gauge}. 

\section{Open Problems}
Nonexpansiveness is a fairly robust property that carries over to compositions and convex 
combinations of mappings without requiring that those operators share fixed points.  
Our notion of $\alpha$-firm mappings appears to be much more demanding.  
Our development begs the question: is the $\alpha$-firmness 
property preserved in some sense under compositions and 
convex compositions of (pointwise) $\alpha$-firm mappings that do not share common 
fixed points?  The answer to this question has immediate bearing on the analysis of 
simple algorithms like cyclic projections for inconsistent feasibility or coordinate descents 
in nonlinear spaces.  

The other open problem, whose solution was hinted at above, is whether compositions 
and averages of $p$-proximal mappings converge at some rate under reasonable assumptions 
of metric subregularity at fixed points.  The notion of {\em almost } $\alpha$-firm 
nonexpansiveness was used in \cite{LukTamTha18} primarily for the purpose of handling 
projectors onto nonconvex sets, and other prox-mappings of nonconvex functions.  
Since the technology of almost $\alpha$-firmness is required for the $p$-proximal 
mappings of even convex functions,  a study in this direction will also account for
$p$-proximal mappings of nonconvex functions, including projections onto 
nonconvex sets.


\begin{thebibliography}{10}

\bibitem{Alexandrov}
A.~D. Alexandrov.
\newblock A theorem on triangles in a metric space and some of its
  applications.
\newblock {\em Trudy Mat. Inst. Steklova}, 38:5--23, 1951.

\bibitem{AriLeuLop14}
D.~{Ariza-Ruiz}, L.~{Leu\c{s}tean}, and G.~{L\'opez-Acedo}.
\newblock {Firmly nonexpansive mappings in classes of geodesic spaces.}
\newblock {\em {Trans. Am. Math. Soc.}}, 366(8):4299--4322, 2014.

\bibitem{RuiLopNic15}
D.~Ariza-Ruiz, G.~L\'opez-Acedo, and A.~Nicolae.
\newblock The asymptotic behavior of the composition of firmly nonexpansive
  mappings.
\newblock {\em J Optim Theory Appl}, 167:409--429, 2015.

\bibitem{ACL16}
T.~Aspelmeier, C.~Charitha, and D.~R. Luke.
\newblock Local linear convergence of the {ADMM/Douglas–Rachford} algorithms
  without strong convexity and application to statistical imaging.
\newblock {\em SIAM J. Imaging Sci.}, 9(2):842--868, 2016.

\bibitem{BaiBruRei78}
J.~B. Baillon, R.~E. Bruck, and S.~Reich.
\newblock On the asymptotic behavior of nonexpansive mappings and semigroups in
  {B}anach spaces.
\newblock {\em Houston J. Math.}, 4(1):1--9, 1978.

\bibitem{Banert}
S.~Banert.
\newblock Backward–backward splitting in {H}adamard spaces.
\newblock {\em Journal of Math. Anal. and Appl.}, 414(2):656--665, 2014.

\bibitem{BauCom17}
H.~H. {Bauschke} and P.~L. {Combettes}.
\newblock {\em {Convex analysis and monotone operator theory in Hilbert spaces.
  2nd edition.}}
\newblock Cham: Springer, 2nd edition edition, 2017.

\bibitem{Bacak14}
M.~{Ba\v{c}\'ak}.
\newblock {Computing medians and means in Hadamard spaces.}
\newblock {\em {SIAM J. Optim.}}, 24(3):1542--1566, 2014.

\bibitem{Berdellima20}
A.~B\"erd\"ellima.
\newblock {\em {Investigations in Hadamard Spaces}}.
\newblock PhD thesis, Georg-August Universit\"at G\"ottingen, G\"ottingen,
  2020.

\bibitem{Berg}
I.~D. Berg and I.~G. Nikolaev.
\newblock Quasilinearization and curvature of {A}leksandrov spaces.
\newblock {\em Geom. Dedicata}, 133:195--218, 2008.

\bibitem{Bruck73}
R.~E. Bruck.
\newblock Nonexpansive projections on subsets of {B}anach spaces.
\newblock {\em Pacific J. Math.}, 47:341--355, 1973.

\bibitem{Eckstein}
J.~Eckstein.
\newblock {\em Splitting Methods for Monotone Operators with Applications to
  Parallel Optimization}.
\newblock PhD thesis, MIT, Cambridge, MA, 1989.

\bibitem{Ede72}
M.~{Edelstein}.
\newblock {The construction of an asymptotic center with a fixed-point
  property.}
\newblock {\em {Bull. Am. Math. Soc.}}, 78:206--208, 1972.

\bibitem{GoeRei84}
K.~{Goebel} and S.~{Reich}.
\newblock {\em {Uniform convexity, hyperbolic geometry, and nonexpansive
  mappings.}}, volume~83.
\newblock Marcel Dekker, Inc., New York, NY, 1984.

\bibitem{Gromov}
M.~Gromov.
\newblock {CAT($\kappa$)}-spaces: construction and concentration.
\newblock {\em Zap. Nauch. Sem. POMI}, 280:101--140, 2001.

\bibitem{HerLukStu19a}
N.~Hermer, D.~R. Luke, and A.~Sturm.
\newblock Random function iterations for consistent stochastic feasibility.
\newblock {\em Numer. Funct. Anal. Opt.}, 40(4):386--420, 2019.

\bibitem{Ioffe11}
A.~D. Ioffe.
\newblock Regularity on a fixed set.
\newblock {\em SIAM J. Optim.}, 21(4):1345--1370, 2011.

\bibitem{Ioffe13}
A.~D. Ioffe.
\newblock Nonlinear regularity models.
\newblock {\em Math. Program.}, 139(1-2):223--242, 2013.

\bibitem{Izuchukwu}
C.~{Izuchukwu}, G.~C. {Ugwunnadi}, O.~T. {Mewomo}, A.~R. {Khan}, and
  M.~{Abbas}.
\newblock {Proximal-type algorithms for split minimization problem in
  P-uniformly convex metric spaces}.
\newblock {\em {Numer. Algorithms}}, 82(3):909--935, 2019.

\bibitem{Jost95}
J.~{Jost}.
\newblock {Convex functionals and generalized harmonic maps into spaces of non
  positive curvature}.
\newblock {\em {Comment. Math. Helv.}}, 70(4):659--673, 1995.

\bibitem{Jost97}
J.~Jost.
\newblock {\em Nonpositive Curvature: Geometric and Analytic Aspects}.
\newblock Lectures in Mathematics. ETH Zurich. Birkh\"auser, Basel, 1997.

\bibitem{KohLopNic19}
U.~{Kohlenbach}, G.~{L\'opez-Acedo}, and A.~{Nicolae}.
\newblock {Moduli of regularity and rates of convergence for Fej\'er monotone
  sequences}.
\newblock {\em {Isr. J. Math.}}, 232(1):261--297, 2019.

\bibitem{KopRei09}
Eva {Kopeck\'a} and Simeon {Reich}.
\newblock {Asymptotic behavior of resolvents of coaccretive operators in the
  Hilbert ball}.
\newblock {\em {Nonlinear Anal., Theory Methods Appl., Ser. A, Theory
  Methods}}, 70(9):3187--3194, 2009.

\bibitem{Kuwae2014}
K.~{Kuwae}.
\newblock {Jensen's inequality on convex spaces}.
\newblock {\em {Calc. Var. Partial Differ. Equ.}}, 49(3-4):1359--1378, 2014.

\bibitem{Kuwae2015}
K.~{Kuwae}.
\newblock {Resolvent flows for convex functionals and \(p\)-harmonic maps}.
\newblock {\em {Anal. Geom. Metr. Spaces}}, 3:46--72, 2015.

\bibitem{LanPavSch00}
U.~{Lang}, B.~{Pavlovi\'c}, and V.~{Schroeder}.
\newblock {Extensions of Lipschitz maps into Hadamard spaces.}
\newblock {\em {Geom. Funct. Anal.}}, 10(6):1527--1553, 2000.

\bibitem{Leu10}
L.~{Leu\c{s}tean}.
\newblock {Nonexpansive iterations in uniformly convex \(W\)-hyperbolic
  spaces.}
\newblock In {\em {Nonlinear analysis and optimization I. Nonlinear analysis. A
  conference in celebration of Alex Ioffe's 70th and Simeon Reich's 60th
  birthdays, Haifa, Israel, June 18--24, 2008}}, pages 193--210. Providence,
  RI: American Mathematical Society (AMS); Ramat-Gan: Bar-Ilan University,
  2010.

\bibitem{Lim76}
T.~C. Lim.
\newblock Remarks on some fixed point theorems.
\newblock {\em Proc. Am. Math. Soc.}, 60:179--182, 1976.

\bibitem{LukMar20}
D.~R. Luke and A.-L. Martins.
\newblock Convergence analysis of the relaxed {D}ouglas-{R}achford algorithm.
\newblock {\em SIAM J. Opt.}, 30(1):542--584, 2020.

\bibitem{LukSabTeb19}
D.~R. Luke, S.~Sabach, and M.~Teboulle.
\newblock Optimization on spheres: Models and proximal algorithms with
  computational performance comparisons.
\newblock {\em SIAM J. Math. Data Sci.}, 1(3):408--445, 2019.

\bibitem{LukTebTha18}
D.~R. Luke, M.~Teboulle, and N.~H. Thao.
\newblock Necessary conditions for linear convergence of iterated expansive,
  set-valued mappings.
\newblock {\em Math. Program.}, 180:1--31, 2018.

\bibitem{LukTamTha18}
D.~R. Luke, N.~H. Thao, and M.~K. Tam.
\newblock Quantitative convergence analysis of iterated expansive, set-valued
  mappings.
\newblock {\em Math. Oper. Res.}, 43(4):1143--1176, 2018.

\bibitem{NaoSil11}
A.~{Naor} and L.~{Silberman}.
\newblock {Poincar\'e inequalities, embeddings, and wild groups.}
\newblock {\em {Compos. Math.}}, 147(5):1546--1572, 2011.

\bibitem{Ohta07}
S.~{Ohta}.
\newblock {Convexities of metric spaces.}
\newblock {\em {Geom. Dedicata}}, 125:225--250, 2007.

\bibitem{ReiSha87}
S.~{Reich} and I.~{Shafrir}.
\newblock {The asymptotic behavior of firmly nonexpansive mappings.}
\newblock {\em {Proc. Am. Math. Soc.}}, 101:246--250, 1987.

\bibitem{ReiSha90}
S.~{Reich} and I.~{Shafrir}.
\newblock {Nonexpansive iterations in hyperbolic spaces.}
\newblock {\em {Nonlinear Anal., Theory Methods Appl.}}, 15(6):537--558, 1990.

\bibitem{VA}
R.~T. Rockafellar and R.~J. Wets.
\newblock {\em Variational Analysis}.
\newblock Grundlehren Math. Wiss. Springer-Verlag, Berlin, 3 edition, 2009.

\bibitem{Sturm03}
K.-T. {Sturm}.
\newblock {Probability measures on metric spaces of nonpositive curvature}.
\newblock In {\em {Heat kernels and analysis on manifolds, graphs, and metric
  spaces. Lecture notes from a quarter program on heat kernels, random walks,
  and analysis on manifolds and graphs, April 16--July 13, 2002, Paris,
  France}}, pages 357--390. Providence, RI: American Mathematical Society
  (AMS), 2003.

\end{thebibliography}

\end{document}